\documentclass[a4paper,12pt,english,francais]{amsart}

\usepackage{hyperref}

\usepackage{amssymb,babel}
\usepackage[T1]{fontenc}\usepackage{lmodern,calrsfs}

\newcommand{\setC}{\mathbb{C}}
\newcommand{\setN}{\mathbb{N}}
\newcommand{\setR}{\mathbb{R}}
\newcommand{\setZ}{\mathbb{Z}}

\newcommand{\bR}{\mathbf{R}}
\newcommand{\cC}{\mathcal{C}}
\newcommand{\cE}{\mathcal{E}}
\newcommand{\cH}{\mathcal{H}}
\newcommand{\cO}{\mathcal{O}}
\newcommand{\cR}{\mathcal{R}}

\newcommand{\cV}{\mathcal{V}}

\newcommand{\hg}{\gimel}
\newcommand{\hh}{\hat h}

\newcommand{\rR}{\overset{\circ}R}

\renewcommand{\leq}{\leqslant}
\renewcommand{\geq}{\geqslant}

\newcommand{\eh}{\mathrm{eh}}
\newcommand{\EH}{\mathrm{EH}}
\newcommand{\euc}{\mathrm{e}}

\DeclareMathOperator{\tr}{tr}
\DeclareMathOperator{\vol}{vol}

\DeclareMathOperator{\Ind}{Ind}
\DeclareMathOperator{\Ric}{Ric}
\DeclareMathOperator{\Scal}{Scal}
\DeclareMathOperator{\Sp}{Sp}
\DeclareMathOperator{\Spin}{Spin}
\DeclareMathOperator{\SO}{SO}
\DeclareMathOperator{\SU}{SU}
\DeclareMathOperator{\U}{U}
\DeclareMathOperator{\Sym}{Sym}

\newtheorem{theo}{Th\'eor\`eme}[section]
\newtheorem{prop}[theo]{Proposition}
\newtheorem{lemm}[theo]{Lemme}
\newtheorem{coro}[theo]{Corollaire}

\theoremstyle{definition}
\newtheorem{defi}[theo]{D\'efinition}

\theoremstyle{remark}
\newtheorem{rema}[theo]{Remarque}
\newtheorem*{ques}{Question}
\newtheorem*{ssques}{Sous-question}
\newtheorem*{nota}{Notation}

\begin{document}

\author{Olivier Biquard}
\title{D\'esingularisation de m\'etriques d'Einstein. I.}
\address{UPMC Universit\'e Paris 6 et \'Ecole Normale Sup\'erieure, UMR 8553 du CNRS}

\selectlanguage{english}
\begin{abstract}
  We find a new obstruction for a real Einstein 4-orbifold
  with an $A_1$-singularity to be a limit of smooth Einstein
  4-manifolds. The obstruction is a curvature condition at the
  singular point.

  For asymptotically hyperbolic metrics, with boundary at infinity a
  conformal metric, we prove that if the obstruction vanishes, one can
  desingularize Einstein orbifolds with such singularities.

  The Dirichlet problem consists in finding Einstein metrics with
  given conformal infinity on the boundary: we prove that our
  obstruction defines a wall in the space of conformal metrics on the
  boundary, and that all the Einstein metrics must have their
  conformal infinity on one side of the wall.
\end{abstract}

\maketitle

\selectlanguage{francais}
Cet article est consacr\'e aux limites de m\'etriques d'Einstein en dimension
$4$. Par les th\'eor\`emes d'Anderson et Bando-Kasue-Nakajima
\cite{And89,BanKasNak89}, une suite $(g_i)$ de m\'etriques d'Einstein sur des
vari\'et\'es compactes de dimension $4$, \`a volume minor\'e et diam\`etre born\'e, et
norme $L^2$ de la courbure born\'ee, converge vers une m\'etrique d'Einstein
sur un orbifold ; en outre, pr\`es des singularit\'es qui apparaissent, une
limite d'\'eclatements $(\frac{g_i}{\epsilon_i})$, o\`u $\epsilon_i$ tend vers $0$, converge
vers une vari\'et\'e Ricci plate asymptotiquement localement euclidienne (ALE),
c'est-\`a-dire model\'ee \`a l'infini sur le quotient de $\setR^4$ par un sous-groupe
fini de $\SO(4)$.

L'id\'ee g\'en\'erale dans laquelle s'inscrit cet article est de raffiner le
th\'eor\`eme de compacit\'e en comprenant les limites possibles des suites de
m\'etriques d'Einstein. C'est une question importante en g\'eom\'etrie ou pour
les applications des techniques g\'eom\'etriques en topologie. Plus
pr\'ecis\'ement, nous \'etudions les m\'etriques d'Einstein sur des orbifolds pour
lesquelles on peut reconstruire les suites de m\'etriques d'Einstein lisses
convergeant vers la m\'etrique orbifold. Si une telle construction est connue
en g\'eom\'etrie k\"ahl\'erienne---l'exemple de certaines surfaces K3 construites
par d\'esingularisation de surfaces de Kummer singuli\`eres---, ce n'est pas le
cas en g\'eom\'etrie riemannienne, car l'\'equation d'Einstein r\'eelle est
difficile \`a manipuler : en particulier, elle est toujours obstru\'ee dans
cette situation.

Les m\'etriques ALE Ricci-plates ne sont pas classifi\'ees, mais on suppose
souvent qu'elles se r\'eduisent aux instantons gravitationnels ALE,
classifi\'es par Kronheimer \cite{Kro89a,Kro89b}. Ce sont des m\'etriques ALE
hyperk\"ahl\'eriennes, d\'esingularisations ou lissifications des singularit\'es
kleiniennes $\setC^2/\Gamma$, o\`u $\Gamma$ est un sous-groupe fini de $\SU(2)$. Leur
2-homologie est engendr\'ee par les classes d'une configuration de sph\`eres
d'auto-intersection $-2$, contract\'ees sur la singularit\'e dans $\setC^2/\Gamma$. Le
cas le plus simple est donc celui de la singularit\'e $\setC^2/\setZ_2$, dans
laquelle est contract\'ee la seule sph\`ere de sa d\'esingularisation, la
m\'etrique de Eguchi-Hanson d\'efinie sur $T^*S^2$.

C'est ce cas que nous \'etudions dans cet article : le cas g\'en\'eral d'une
singularit\'e kleinienne, nettement plus technique, sera trait\'e dans un autre
article. La singularit\'e $\setC^2/\setZ_2$ est aussi le cas le plus \guillemotleft~g\'en\'erique~\guillemotright{} de
convergence vers une singularit\'e kleinienne, au sens o\`u l'on ne contracte
qu'une seule sph\`ere d'auto-intersection $-2$. Une autre limitation de cet
article est que nous ne r\'esolvons la question que dans un cadre non
compact, o\`u un bout \`a l'infini est asymptotiquement hyperbolique (AH)---le
cas compact posant les probl\`emes suppl\'ementaires connus d'obstruction dans
les \'equations d'Einstein. N\'eanmoins, une grande partie de notre
construction est valable dans le cas compact et y soul\`eve des questions
int\'eressantes. Elle \'eclaire aussi le probl\`eme de Dirichlet \`a l'infini sur
les vari\'et\'es AH, \'etudi\'e par Anderson, dans le cas o\`u la vari\'et\'e poss\`ede un
$H^2$ qui ne provient pas de son bord.

Passons maintenant aux r\'esultats de cet article. On part d'un orbifold
$M_0$, \`a bord $\partial M_0=\{x=0\}$ (o\`u $x$ est une \'equation du bord), muni d'une
m\'etrique d'Einstein AH $g_0$, au sens o\`u pr\`es du bord $\partial M_0$ on a $g_0\sim
\frac{dx^2+\gamma_0}{x^2}$, o\`u $\gamma_0$ est une m\'etrique sur $\partial M_0$. Le bord $\partial M_0$
appara\^\i t donc comme bord \`a l'infini, sur lequel $g_0$ n'induit
intrins\`equement que la classe conforme $[\gamma_0]$, appel\'ee infini conforme de
$g_0$. Le but est de d\'esingulariser la m\'etrique aux points orbifold, ce qui
sera possible sous une hypoth\`ese de courbure. Pour l'\'enoncer, rappelons que
la courbure riemannienne peut \^etre vue comme endomorphisme sym\'etrique des
2-formes, et qu'en dimension 4, celles-ci se d\'ecomposent en formes
autoduales et anti-autoduales : $$\Omega^2=\Omega^2_+\oplus\Omega^2_-.$$ La courbure se d\'ecompose
alors en
$$ \bR =
\begin{pmatrix}
  \bR^+ & \Ric^0 \\ \Ric^0 & \bR^-
\end{pmatrix}, \quad \bR^\pm = W^\pm + \frac{\Scal}{12},$$ o\`u $W^+$ et $W^-$, \`a
trace nulle, sont les parties autoduales et anti-autoduales du tenseur de
Weyl $W$.

Enfin, pr\'ecisons que $g_0$ est dite non d\'eg\'en\'er\'ee si elle n'admet pas de
d\'eformation d'Einstein infinit\'esimale $L^2$.

\begin{theo}\label{thm:principal}
  Supposons que $(M_0,g_0)$ ait une singularit\'e de la forme $\setR^4/\setZ_2$ au
  point $p_0$, et que $g_0$ soit non d\'eg\'en\'er\'ee. Soit $M$ la vari\'et\'e
  obtenue \`a partir de $M_0$ en \'eclatant $p_0$ en une sph\`ere
  d'auto-intersection $-2$. Supposons en outre qu'au point $p_0$ on ait
  \begin{equation}
  \det \bR^+_{g_0}(p_0)=0.\label{eq:1}
  \end{equation}
  Alors il existe une famille $g_t$ ($t>0$) de m\'etriques d'Einstein AH
  sur $M$, telle que
  \begin{itemize}
  \item le volume de la sph\`ere \'eclatant $p_0$ est $t$ ;
  \item quand $t\to0$, on a $g_t\to g_0$, avec convergence $C^\infty$ sur tout
    compact en dehors de $p_0$ ;
  \item \`a la singularit\'e, une bulle se forme : quand $t\to0$, on a
    $\frac{g_t}{t}\to\eh$.
  \end{itemize}

  Si $M_0$ compte plusieurs singularit\'es de la forme $\setR^4/\setZ_2$, le
  m\^eme r\'esultat reste vrai : on peut d\'esingulariser ou bien toutes les
  singularit\'es simultan\'ement, ou bien seulement un sous-ensemble en
  laissant les singularit\'es orbifold aux autres points.
\end{theo}
Un exemple explicite de d\'esingularisation est connu : la famille de
m\'etriques AdS-Taub-Bolt, construite par Page-Pope \cite{PagPop87}, voir
aussi \cite{MazSin07}. C'est une famille \`a 1 param\`etre de m\'etriques
d'Einstein AH sur $T^*S^2$, d\'eg\'en\'erant vers une m\'etrique orbifold
$g_0$. Cette m\'etrique $g_0$ v\'erifie l'hypoth\`ese de courbure $\det
\bR^+(p_0)=0$. En fait, elle co\"\i ncide avec le quotient par $\setZ_2$ d'une
m\'etrique de Pedersen \cite{Ped86}, et en particulier est autoduale ; par
\cite[prop. 4.5.3]{BiqRol09}, elle est non d\'eg\'en\'er\'ee.

La m\'ethode de d\'emonstration du th\'eor\`eme consiste \`a recoller la m\'etrique de
Eguchi-Hanson de $T^*S^2$, asymptote \`a $\setR^4/\setZ_2$, \`a la m\'etrique orbifold
apr\`es un changement d'\'echelle d\'ependant de $t$. Dans ce recollement
appara\^\i t un param\`etre de jauge, $\varphi\in \SO(4)$, indiquant comment la carte \`a
l'infini de Eguchi-Hanson est recoll\'ee \`a une carte de $M_0$ au point
singulier. Comme la m\'etrique de Eguchi-Hanson est $\U(2)$-invariante,
le param\`etre $\varphi$ habite dans $\SO(4)/\U(2)$, qui s'identifie aux directions
de $\Omega^2_+\setR^4$.

Choisissons sur Eguchi-Hanson la carte \`a l'infini qui est complexe pour la
structure complexe de $T^*P^1_\setC$. Alors, dans notre m\'ethode de recollement,
il faut que la droite de $\Omega^2_+\setR^4$ d\'etermin\'ee par $\varphi$ soit dans le noyau
de $\bR^+(p_0)$. Dans le cas non d\'eg\'en\'er\'e o\`u une seule valeur propre de
$\bR^+(p_0)$ s'annule, la jauge de recollement est donc compl\`etement
d\'etermin\'ee ; dans le cas d\'eg\'en\'er\'e o\`u deux valeurs propres s'annulent, il y
a plus de libert\'e.

Cela correspond bien \`a ce qui se passe dans le cas k\"ahl\'erien : les valeurs
propres de $\bR^+$ sont alors $(\frac{\Scal}4,0,0)$ ; si la courbure
scalaire de $M_0$ est non nulle, la structure complexe de $M_0$ ne peut pas
\^etre recoll\'ee \`a celle de $T^*P^1_\setC$, la m\'etrique K\"ahler-Einstein ne pouvant
pas admettre de courbe holomorphe d'auto-intersection $-2$. Il faut donc la recoller \`a
une autre structure complexe de Eguchi-Hanson, pour laquelle la section
nulle n'est pas holomorphe.

Enfin, si la m\'etrique orbifold satisfait plut\^ot $\det \bR^-_{g_0}(p_0)=0$ en
l'un des points singuliers, le th\'eor\`eme reste valable en recollant en
$p_0$ une m\'etrique de Eguchi-Hanson avec l'orientation oppos\'ee. Par
cons\'equent, compte tenu de la forme diagonale de $\bR$ pour une
m\'etrique d'Einstein, on peut consid\'erer que la vraie condition de
courbure pour rendre possible le recollement est
$$ \det \bR_{g_0}(p_0)=0, $$
et une droite du noyau d\'etermine un param\`etre de recollement dans
$\mathrm{O}(4)/\U(2)$, c'est-\`a-dire \`a la fois un choix d'orientation
et un choix de structure complexe dans l'espace de Eguchi-Hanson recoll\'e.

\medskip Une partie de la d\'emonstration du th\'eor\`eme \ref{thm:principal} est
valable sur les vari\'et\'es compactes. Cela permet de penser que, $M_0$
compacte ou non compacte, la condition de courbure $\det \bR^+(p_0)=0$ soit
n\'ecessaire pour que $g_0$ soit la limite d'une suite $(g_i)$ de m\'etriques
d'Einstein convergeant vers $g_0$, telle que $(\frac{g_i}{\epsilon_i})$ converge
vers la m\'etrique de Eguchi-Hanson. Voici un premier r\'esultat dans cette
direction. Pour l'\'enoncer, soit $g_t^0$ la m\'etrique qu'on obtient en
recollant une grande boule de rayon $t^{-\frac14}$, contract\'ee d'un facteur
$t^{\frac12}$, dans une petite boule de rayon $t^{\frac14}$ de $M_0$. Il
s'agit du recollement le plus simple possible entre $g_0$ et la m\'etrique de
Eguchi-Hanson qui d\'esingularise la singularit\'e. Soit $\rho\geq 1$ une fonction,
qui satisfait :
\begin{itemize}
\item en dehors d'un compact de Eguchi-Hanson, $\rho$ co\"\i ncide avec le rayon ;
\item sur $M_0$, dans une boule fixe $B$ autour de la singularit\'e, la
  fonction $\rho$ co\"\i ncide avec $t^{-\frac12}r$, o\`u $r$ est le noyau ;
\item dans $M_0-B$, la fonction $\rho$ est tronqu\'ee par une constante (donc de
  l'ordre de $t^{-\frac12})$.
\end{itemize}
Il y a bien s\^ur un recollement sur la zone de transition entre les deux
m\'etriques. Alors :
\begin{theo}\label{theo:necessite}
  Supposons que $(M_0,g_0)$ soit un orbifold compact, d'Einstein, non
  d\'eg\'en\'er\'e, avec singularit\'e de la forme $\setR^4/\setZ_2$ au point
  $p_0$. Supposons qu'au point $p_0$ on ait
  $$\det \bR_{g_0}^+(p_0)\neq 0.$$
  Soit $\delta_0>0$ petit. Alors il existe $\epsilon>0$, tel que si $g$ est une
  m\'etrique sur $M$, proche dans $C^2$ de $g_t^0$ au sens o\`u, pour $k=0$,
  $1$, $2$,
  $$ |\nabla^k(g-g_t^0)|_{g_t^0} \leq \epsilon \rho^{-\delta_0-k}, $$
  (en fait il suffit d'avoir cette borne dans $C^{1,\alpha}$),
  alors $g$ n'est pas d'Einstein.
\end{theo}
Voir l'\'enonc\'e technique pr\'ecis $C^{1,\alpha}$ dans le corollaire
\ref{cor:obs-E}, qui traite aussi le cas non compact (un autre poids est
alors n\'ecessaire \`a l'infini). Notons tout de suite que la condition sur la
diff\'erence n'est pas trop forte, au sens o\`u la borne $C^2$ exig\'ee sur
$g-g_t^0$ est loin d'emp\^echer $g$ d'\^etre Einstein par un argument
local. Par ailleurs, la proximit\'e demand\'ee entre $g$ et $g_t^0$ est d'ordre
$t^{\frac{\delta_0}2}$, ce qui est raisonnable puisqu'on s'attend \`a ce qu'une
\'eventuelle m\'etrique d'Einstein diff\`ere en $O(t)$ de la m\'etrique $g_t^0$
(c'est le cas pour les m\'etriques produites par le th\'eor\`eme
\ref{thm:principal}).

On montrera dans un autre article que la condition de courbure (\ref{eq:1})
est n\'ecessaire aussi, dans le m\^eme sens, pour d\'esingulariser les autres
singularit\'es kleiniennes. \'Evidemment il serait souhaitable de d\'emontrer que
la condition est n\'ecessaire sous des hypoth\`eses beaucoup plus faibles sur
la suite de m\'etriques d'Einstein convergeant vers $g_0$ :
\begin{ques}
  Supposons qu'on ait une suite de m\'etriques d'Einstein $(g_i)$,
  convergeant vers un orbifold $g_0$ avec une singularit\'e $\setC^2/\Gamma$ au
  point $p$, et telle qu'un \'eclatement au point $p$, soit
  $\frac{g_i}{\epsilon_i}$, converge vers un instanton gravitationnel ALE. Alors
  a-t-on $\det\bR^+_{g_0}(p_0)=0$ ?
\end{ques}
L'hypoth\`ese de non d\'eg\'en\'erescence est \'evidemment satisfaite pour les
m\'etriques hyperboliques r\'eelles, qui en revanche ne satisfont jamais la
condition de courbure $\det\bR^+(p_0)=0$. Donc le th\'eor\`eme
\ref{theo:necessite} dit qu'on ne peut pas
d\'esingulariser un orbifold hyperbolique r\'eel, avec des singularit\'es d'ordre
2, par la m\'etrique de Eguchi-Hanson. Il est int\'eressant de sp\'ecialiser
la question pr\'ec\'edente aux vari\'et\'es hyperboliques :
\begin{ssques}
  Si $(M_0,g_0)$ est un orbifold compact de dimension 4, hyperbolique r\'eel,
  avec des singularit\'es kleiniennes, est-il possible qu'existe une suite
  $g_i$ de m\'etriques d'Einstein sur des vari\'et\'es compactes, convergeant
  vers $g_0$ en dehors des singularit\'es, et telle qu'un \'eclatement aux
  singularit\'es converge vers un instanton gravitationnel ALE ?
\end{ssques}
La m\^eme question fait \'egalement sens dans le cadre non compact.

\medskip Dans le th\'eor\`eme \ref{thm:principal}, les m\'etriques $g_t$
modifient l'infini conforme de $g_0$, et le comportement de l'infini
conforme par rapport \`a $t$ est d\'etermin\'e. Cela \'eclaire le probl\`eme de
Dirichlet \`a l'infini pour les m\'etriques d'Einstein AH : \emph{\'etant
  donn\'e une m\'etrique conforme $[\gamma]$ sur $\partial M$, trouver une m\'etrique
  d'Einstein AH sur $M$ dont l'infini conforme est $[\gamma]$.} Ce probl\`eme
est \'etudi\'e par Anderson \cite{And05,And08} par une th\'eorie du degr\'e
qui compte le nombre de m\'etriques d'Einstein d'infini conforme donn\'e,
sous l'hypoth\`ese que toute la 2-homologie de $M$ provienne de $\partial M$,
ce qui permet justement d'\'eviter la d\'eg\'en\'erescence vers un
orbifold. La d\'emonstration du th\'eor\`eme \ref{thm:principal} permet de
pr\'eciser ce qui se passe dans ce cas, si les singularit\'es sont
kleiniennes. Notons $\cC$ l'espace des m\'etriques conformes sur $\partial
M_0$. Pour $\gamma\in \cC$, proche de $\gamma_0$, il existe \cite{Biq00} une
m\'etrique d'Einstein AH $g_0(\gamma)$, d'infini conforme $\gamma$.  Soit
\begin{equation}
 \cC_0 = \{ \gamma\in \cC, \det \bR^+_{g_0(\gamma)}(p_0)=0 \}.\label{eq:2}
\end{equation}
Alors le r\'esultat suivant dit que, g\'en\'eriquement, on peut d\'esingulariser
d'un seul c\^ot\'e de $\cC_0$, et qu'ainsi $\cC_0$ est un \guillemotleft~mur~\guillemotright{} pour le
probl\`eme de Dirichlet \`a l'infini ; les \'enonc\'es pr\'ecis se trouvent sections
\ref{sec:asymptotique} et \ref{sec:transversalite} :
\begin{theo}
  Sous une hypoth\`ese \guillemotleft{} g\'en\'erique \guillemotright{} sur $\gamma_0$, l'hypersurface $\cC_0\subset \cC$
  est lisse en $\gamma_0$, et :
  \begin{enumerate}\item 
    les infinis conformes des m\'etriques d\'esingularis\'ees construites
    par le th\'eor\`eme \ref{thm:principal} remplissent l'intersection
    d'un voisinage ouvert de $\gamma_0$ avec un seul c\^ot\'e de $\cC_0$ ;\item
    si une m\'etrique d'Einstein AH $g$ sur $M$, d'infini conforme $\gamma$,
    est proche au sens du th\'eor\`eme \ref{theo:necessite} d'un
    recollement de la m\'etrique de Eguchi-Hanson avec une m\'etrique
    orbifold d'Einstein, alors il faut que $\gamma$ soit du c\^ot\'e de $\cC_0$
    d\'etermin\'e ci-avant, et $g$ est d\'etermin\'ee par $\gamma$.
  \end{enumerate}
\end{theo}
Ce th\'eor\`eme indique manifestement que, dans le cas g\'en\'eral, chaque
sph\`ere d'auto-intersection $-2$ dans la 2-homologie de $M$ donne
naissance \`a un mur d\'efini par (\ref{eq:2}) sur l'orbifold obtenu par
contraction de cette sph\`ere : le degr\'e d'Anderson change \`a travers le
mur, et la contribution des m\'etriques que nous avons construites est
exactement $1$. Pour en d\'eduire que la formule de passage de mur donne
une diff\'erence \'egale \`a $1$, il faudrait en outre montrer qu'il n'y a
pas d'autre moyen de d\'eg\'en\'erer vers $g_0$ \`a travers des m\'etriques
d'Einstein, que par des m\'etriques proches, au sens du th\'eor\`eme
\ref{theo:necessite}, des recollements mod\`eles $g_t^0$. Voir l'exemple
de $T^*S^2$ \`a la fin de la section \ref{sec:transversalite}.

Pour des sph\`eres d'auto-intersection diff\'erente de $-2$, il y a dans
\cite{And08} quelques remarques tendant \`a sugg\'erer que ce type de
ph\'enom\`ene ne se produira pas.

Les autres singularit\'es kleiniennes correspondent \`a certaines
configurations de sph\`eres d'auto-intersection $-2$. On peut aussi les
d\'esingulariser, mais les obstructions \`a annuler portent non seulement sur
$\bR^+$, mais aussi sur les d\'eriv\'ees de la courbure. Ce probl\`eme beaucoup
plus technique sera trait\'e dans un autre article. Par rapport \`a la question
ci-dessus, on s'attend \`a trouver les d\'eg\'en\'erescences vers ces singularit\'es
plus compliqu\'ees aux intersections des diverses hypersurfaces $\cC_0$
d\'efinies par les diff\'erentes sph\`eres d'auto-intersection $-2$. On voit
donc que la codimension sera sup\'erieure \`a $1$, et qu'il s'agit d'un cas
beaucoup moins g\'en\'erique que celui trait\'e dans cet article.

\medskip Passons \`a pr\'esent au plan de cet article. On commence par analyser
en profondeur l'\'equation infinit\'esimale d'Einstein pour la m\'etrique de
Eguchi-Hanson en \S~\ref{sec:la-linearisation-de}. Pour recoller la m\'etrique
de Eguchi-Hanson, apr\`es changement d'\'echelle, \`a la m\'etrique orbifold donn\'ee
$g_0$, on modifie en \S~\ref{sec:le-probl-line} la m\'etrique de Eguchi-Hanson
pour la faire coller \`a l'infini avec les termes d'ordre 2 de $g_0$ au point
singulier. C'est possible sous une condition sur ses termes d'ordre 2,
interpr\'et\'ee en \S~\ref{sec:reform-des-obstr} comme condition de courbure sur
$g_0$ au point singulier. Cette condition permet donc de fabriquer en
\S~\ref{sec:constr-dune-solut} un recollement qui est approximativement
d'Einstein. Malheureusement, on ne peut pas le d\'eformer en une vraie
m\'etrique d'Einstein, \`a cause des obstructions provenant de Eguchi-Hanson,
et on est amen\'e \`a d\'efinir en \S~\ref{sec:Einstein-mod-O} un probl\`eme
d'Einstein modulo obstructions. On a \'et\'e concis dans la r\'esolution,
effectu\'ee en \S~\ref{sec:espaces-fonctionnels}--\ref{sec:resol-mod-O}, car
elle suit un sch\'ema g\'en\'eral classique : n\'eanmoins la d\'emonstration reste
d\'elicate, car il faut d'une part contr\^oler une g\'eom\'etrie dont le rayon
d'injectivit\'e tend vers $0$, et d'autre part obtenir un d\'eveloppement
pr\'ecis des obstructions. De ce d\'eveloppement est d\'eduit en particulier le
corollaire \ref{cor:obs-E} qui dit que, sous certaines hypoth\`eses, la
condition de courbure $\det\bR^+_{g_0}(p_0)=0$ est n\'ecessaire. Jusque l\`a,
toute la construction, y compris ce corollaire, reste valable si $M_0$ est
compacte.

La r\'esolution compl\`ete de l'\'equation exige de trouver de nouvelles
d\'eformations pour tuer les obstructions : celles-ci proviennent des
d\'eformations de l'infini conforme de $M_0$ (ce n'est donc qu'ici qu'on
utilise la non compacit\'e de $M_0$), permettant de perturber le germe
de $g_0$ au point singulier en
\S~\ref{sec:solutions-orbifold}--\ref{sec:extension-de-germes} (ceci
repose sur le th\'eor\`eme de continuation unique de \cite{Biq08}), puis
on construit les m\'etriques d'Einstein en
\S~\ref{sec:metriques-E}. D\'eterminer l'asymptotique pr\'ecis des infinis
conformes n\'ecessaires requiert un raffinement des obstructions et des
solutions approch\'ees men\'e en
\S~\ref{sec:petit-vale-propr}--\ref{sec:asymptotique}. Finalement, un
argument de transversalit\'e d\'evelopp\'e en \S~\ref{sec:transversalite}
donne l'unicit\'e g\'en\'erique des solutions.

\section{La lin\'earisation de l'\'equation d'Einstein pour Eguchi-Hanson}
\label{sec:la-linearisation-de}

La m\'etrique de Eguchi-Hanson, $\eh$, est une m\'etrique hyperk\"ahl\'erienne
sur $T^*P^1_\setC$. \`A l'infini, elle est ALE (asymptotiquement
localement euclidienne), c'est-\`a-dire asymptotique \`a la m\'etrique plate
de $\setR^4/\setZ_2$.  On rappelle ici quelques faits standards.
Choisissons des coordonn\'ees $(x^1,x^2,x^3,x^4)$ sur $\setR^4$, soient
$r$ le rayon et $(\alpha_1,\alpha_2,\alpha_3)$ une base de 1-formes
invariantes sur la sph\`ere $S^3$, satisfaisant
$d\alpha_1=2\alpha_2\land \alpha_3, \dots$ Plus pr\'ecis\'ement, on
prend
\begin{align*}
  \alpha_1&=\tfrac1{r^2}(x^1dx^2-x^2dx^1+x^3dx^4-x^4dx^3), \\
   \alpha_2&=\tfrac1{r^2}(x^1dx^3-x^3dx^1-x^2dx^4+x^4dx^2), \\
   \alpha_3&=\tfrac1{r^2}(x^1dx^4-x^4dx^1+x^2dx^3-x^3dx^2).
\end{align*}
On peut alors \'ecrire la m\'etrique de Eguchi-Hanson comme
\begin{equation}
  \label{eq:3}
  \eh = \frac{r^2}{\sqrt{1+r^4}}(dr^2+r^2\alpha_1^2)+\sqrt{1+r^4}(\alpha_2^2+\alpha_3^2),
\end{equation}
o\`u la section nulle de $T^*P^1_\setC$ est repr\'esent\'ee par $r=0$, et la
m\'etrique y est \'equivalente \`a
$$ d(\tfrac{r^2}2)^2+4(\tfrac{r^2}2)^2\alpha_1^2+\alpha_2^2+\alpha_3^2. $$
Pour obtenir une m\'etrique qui s'\'etende de mani\`ere lisse \`a l'origine, on
voit qu'il faut que $2\alpha_1$ soit de longueur $2\pi$ sur les
cercles autour du $P^1$, c'est-\`a-dire qu'il faut quotienter par
l'action de $\setZ_2$ sur ces cercles ; la formule (\ref{eq:3})
d\'efinit ainsi une m\'etrique lisse sur $T^*P^1_\setC$.

Comme m\'etrique hyperk\"ahl\'erienne, Eguchi-Hanson est k\"ahl\'erienne par
rapport \`a trois structures complexes $I_1$, $I_2$ et $I_3$ ; on
choisira $I_1$ comme la structure complexe donnant $T^*P^1_\setC$, elle est
donn\'ee explicitement par les coordonn\'ees complexes $x^1+ix^2$ et
$x^3+ix^4$. Les structures complexes $I_2$ et $I_3$ sont asymptotes, \`a
l'infini, aux structures complexes de $\setR^4$ donn\'ees
respectivement, d'une part par les coordonn\'ees complexes $x^1+ix^3$ et
$x^2-ix^4$, d'autre part par les coordonn\'ees complexes $x^1+ix^4$ et
$x^2+ix^3$.

Enfin, \`a l'infini, notant la m\'etrique euclidienne
$$ \euc = dr^2+r^2(\alpha_1^2+\alpha_2^2+\alpha_3^2), $$
on a le d\'eveloppement asymptotique
\begin{equation}
  \label{eq:4}
  \eh = \euc +\frac1{2r^4}\big(-(dr^2+r^2\alpha_1^2)+\alpha_2^2+\alpha_3^2\big)
       +O(\tfrac1{r^6}),
\end{equation}
qui montre que la m\'etrique de Eguchi-Hanson est asymptotique \`a la
m\'etrique plate de $\setC^2/\setZ_2$ \`a des termes en $O(\frac1{r^4})$
pr\`es.

La cohomologie $L^2$ de $\eh$, concentr\'ee en degr\'e 2, est de dimension 1 et
repr\'esent\'ee par la 2-forme anti-autoduale ferm\'ee
\begin{equation}
  \label{eq:5}
  \Omega = \frac1{1+r^4}\left(\frac{r^2}{\sqrt{1+r^4}}dr\land r\alpha_1-\sqrt{1+r^4}\alpha_2\land\alpha_3\right).
\end{equation}

L'op\'erateur de lin\'earisation de l'\'equation d'Einstein, agissant sur les
2-tenseurs sym\'etriques, est
$$ d_\eh \Ric(h)=\frac12 \nabla_\eh ^*\nabla_\eh  h - \rR_\eh  h - \delta_\eh ^* B_\eh  h , $$
o\`u $B_\eh =\delta_\eh  h+\frac12d\tr_\eh  h$ est l'op\'erateur de Bianchi.

\emph{Dans la suite, sauf ambigu\"\i t\'e, on \'evitera les indices indiquant par
  rapport \`a quelle m\'etrique est d\'efini l'op\'erateur.}

G\'en\'eralement, on consid\`ere $h$ en jauge de Bianchi $Bh=0$, ce qui m\`ene \`a
\'etudier plut\^ot
$$ P = \frac12 \nabla^*\nabla - \rR . $$
Puisque $\Ric_\eh =0$, l'identit\'e de Bianchi $B_g \Ric_g=0$ implique
$B_\eh \circ d_\eh \Ric=0$, d'o\`u on d\'eduit imm\'ediatement
$$ B_\eh  P_\eh  = B_\eh  \delta_\eh ^* B_\eh  = \frac12 \nabla_\eh ^*\nabla_\eh  B_\eh  . $$
Une cons\'equence de cette formule est le fait suivant : si $P_\eh h=0$
et $h=O(r^\delta)$ pour un $\delta<1$, alors $B_\eh h=0$. Par cons\'equent, le
noyau de $P$, et particuli\`erement son noyau $L^2$, param\`etre bien des
d\'eformations infinit\'esimales d'Einstein.

\begin{prop}\label{prop:ker-P-EH}
  L'op\'erateur $P_\eh $ pour la m\'etrique de Eguchi-Hanson satisfait les
  propri\'et\'es suivantes :
  \begin{enumerate}
  \item Le noyau $L^2$ de $P_\eh $, not\'e $\cO_\eh$, est de dimension 3,
    engendr\'e par les trois tenseurs
    $$ o_i(x,y) = \Omega(I_ix,y), $$
    o\`u $\Omega$ est la 2-forme harmonique $L^2$, et les $I_i$ sont les
    trois structures complexes de la m\'etrique de Eguchi-Hanson.
  \item Le noyau de $P_\eh $ sur les tenseurs $u$ tels que $|u| =O(r^{-\delta})$
    pour un $\delta>0$ se r\'eduit \`a $\cO_\eh$.
  \item Si $v=O(r^{-\delta})$ pour un $\delta\in(0,4)$, $\delta\neq 2$, alors il existe
    $u$ tel que
    $$ P_\eh u=v, \quad u=O(r^{2-\delta})$$
    si et seulement si le produit scalaire $(v,o_i)=0$ pour $1\leq i\leq 3$.
  \item Une solution $\hat k_i$ de $P_\eh\hat k_i=o_i$ est donn\'ee par
    $$ \hat k_i = -\frac1{12} (1+r^4)^{\frac32} o_i. $$
  \item Une solution $k_i$, \`a trace nulle, de $d_\eh\Ric k_i=o_i$ est donn\'ee par
    \begin{align*}
      k_1 &= \frac32 \hat k_1-\frac18 \sqrt{1+r^4}\eh
            +\frac1{12}\delta^*\big(\frac{(1+r^4)^{\frac32}-1}{r^3}\partial_r\big) , \\
      k_i &= 2 \hat k_i, \quad i=2,3.
    \end{align*}
  \end{enumerate}
\end{prop}
\begin{proof}
  On a $\tr P(h) = \Delta \tr h$, donc si $P(h)=0$, il faut que $\tr h=0$. On
  est donc ramen\'e au cas sans trace. On utilise alors l'isomorphisme
  $\Sym_0^2=\Omega_-^2\Omega_+^2$ pour identifier $h$ avec une 2-forme anti-autoduale
  \`a valeurs dans les 2-formes autoduales. Par la formule (4.6) dans
  \cite{BiqRol09}, on a alors la formule
  \begin{equation}
 P h = d_- d_-^* h ,\label{eq:6}
\end{equation}
  o\`u $d_-:\Omega^1\Omega_+^2 \to\Omega_-^2\Omega_+^2$ est la diff\'erentielle ext\'erieure. Comme la
  m\'etrique est hyperk\"ahl\'erienne, le fibr\'e $\Omega_+^2$ est trivial, et trivialis\'e
  par les trois structures complexes parall\`eles $I_i$ : l'identification de
  $\Omega_-^2\Omega_+^2$ avec $\Omega_-^2\otimes\setR^3$ qui en r\'esulte est donn\'ee par
  $h\mapsto\big(\omega_i(x,y)=h(x,I_iy)\big)$. Ainsi l'op\'erateur
  $d_-d_-^*$ devient le laplacien ordinaire sur les $\omega_i$.
  Comme la cohomologie $L^2$ est de dimension 1, on obtient la premi\`ere
  assertion.

  Pour la seconde assertion, observons que si, pour une 2-forme
  anti-autoduale $\omega$ telle que $|\omega| = O(r^{-\delta})$ on a $d_-d_-^*\omega=0$, alors
  par r\'egularit\'e elliptique $|\nabla\omega| =O(r^{-\delta-1})$ ; par cons\'equent, dans
  l'int\'egration par parties
$$ \int_{r<R} \langle d_-d_-^*\omega,\omega\rangle = \int_{r<R} |d_-^*\omega|^2 - \int_{r=R} d_-^*\omega\land \omega $$
le terme de bord est $O(r^{-\delta-1}r^{-\delta}r^3)=O(r^{2(1-\delta)})$, et donc, si
$\delta>1$ on obtient $d_-^*\omega=0$, c'est-\`a-dire $d\omega=0$. On en d\'eduit alors
facilement qu'en r\'ealit\'e $\omega$ doit co\"\i ncider avec un multiple de la forme
harmonique $\Omega$.

  On utilise alors la th\'eorie des op\'erateurs elliptiques dans les espaces \`a
  poids pour constater que l'op\'erateur $d_-d_-^*$ n'a pas de poids critique
  dans l'intervalle $(0,2)$, et donc le noyau dans cet intervalle ne d\'epend
  pas du poids $\delta\in (0,2)$. Puisque, pour $\delta>1$, il est r\'eduit au noyau
  $L^2$, il est donc r\'eduit au noyau $L^2$ sur tout l'intervalle $(0,2)$

  La troisi\`eme assertion r\'esulte aussi de la th\'eorie des
  op\'erateurs elliptiques dans les espaces \`a poids. En effet, si $\delta$ n'est
  pas critique (d'o\`u la restriction $\delta\neq 2$), le conoyau de l'op\'erateur
  $P:r^{\delta-2}C^{2,\alpha}\to r^\delta C^\alpha$ est \'egal au noyau de $P$ sur $r^{4-\delta}C^\alpha$. Mais
    si $\delta\in (0,4)$ alors $4-\delta\in (0,4)$ et le noyau dans cet intervalle se
    r\'eduit au noyau $L^2$.

  Compte tenu de la formule (\ref{eq:6}), la quatri\`eme
  assertion se ram\`ene \`a l'\'egalit\'e $d_-d_-^*\Xi=\Omega$, o\`u
  $$ \Xi = -\frac1{12}(1+r^4)^{\frac32} \Omega.$$
  Or on calcule
  \begin{equation}
    \label{eq:7}
    d^*\Xi = \frac{r^4}{2\sqrt{1+r^4}}\alpha_1 ,
  \end{equation}
  qui m\`ene \`a
  \begin{equation}
    \label{eq:8}
    dd^*\Xi = \omega_1 + \Omega .
  \end{equation}
  La partie anti-autoduale est donc $\Omega$, d'o\`u la formule.

  La derni\`ere assertion est plus compliqu\'ee : le probl\`eme est que les
  $\hat k_i$ ne sont plus en jauge de Bianchi, donc le calcul de
  $P_\eh\hat k_i$ n'est pas suffisant pour obtenir $d_\eh\Ric \hat
  k_i$. Calculons donc le terme de jauge dans $P_\eh \hat k_i$.
  Observons que la formule (\ref{eq:7}) indique que $d^*\Xi$ est duale
  (au sens de la m\'etrique) du champ de vecteurs de Killing $\frac 12
  I_1 r\frac \partial{\partial r}$. Il en r\'esulte
  \begin{equation}
    \label{eq:9}
    \nabla d^*\Xi = \frac12 dd^*\Xi = \frac12(\omega_1+\Omega).
  \end{equation}
  Compte tenu de $\hat k_i(\cdot,\cdot)=\Xi(I_i\cdot,\cdot)$, on en d\'eduit $\delta\hat
  k_i=I_id^*\Xi$, d'o\`u, \`a l'aide de (\ref{eq:8}),
  \begin{equation*}
    \nabla B\hat k_i(\cdot,\cdot) = -\frac12(\omega_1+\Omega)(\cdot,I_i\cdot).
  \end{equation*}
  En prenant la partie sym\'etrique, on obtient
  \begin{align*}
    \delta^*B\hat k_1 &= \frac12(\eh + o_1), \\
    \delta^*B\hat k_i &= \frac12 o_i, \quad i=2,3.
  \end{align*}
  Puisque $d\Ric=P-\delta^*B$, on en d\'eduit l'assertion sur $k_2$ et
  $k_3$. En revanche, 
  \begin{equation}
    d_\eh\Ric \hat k_1=\frac12(o_1-\eh).\label{eq:10}
  \end{equation}
  Pour compenser le terme $\eh$, on introduit un terme \`a trace : la
  fonction $f(r)=\sqrt{1+r^4}$ satisfait $\Delta f=-8$, et, plus
  pr\'ecis\'ement, $df=-4I_1d^*\Xi$, donc par (\ref{eq:9}),
  $$ \nabla df = 2(\eh + o_1). $$
  Il en r\'esulte
  \begin{align*}
    d_\eh\Ric(f\eh) &= P(f\eh)-\delta^*B(f\eh) \\
                   &= \frac12 (\Delta f) \eh - \delta^*df \\
                   &= -6 \eh -2 o_1 .
  \end{align*}
  Combinant avec (\ref{eq:10}), on en d\'eduit que $d_\eh\Ric \frac32(\hat
  k_1-\frac1{12}f\eh)=o_1$. Le terme de jauge additionnel dans la
  proposition, $\frac1{12}\delta^*\big(\frac{(1+r^4)^{\frac32}-1}{r^3}\partial_r\big)$,
  permet d'annuler la trace (calcul laiss\'e au lecteur).
\end{proof}

\section{Le probl\`eme lin\'earis\'e sur Eguchi-Hanson}
\label{sec:le-probl-line}

Soit dans $\setR^4$ un 2-tenseur sym\'etrique $H$ dont les coefficients sont
des polyn\^omes homog\`enes de degr\'e 2, ainsi
\begin{equation}
H=H_{ijkl}x^ix^jdx^kdx^l.\label{eq:11}
\end{equation}
On pense \`a $H$ comme le terme d'ordre 2 dans le d\'eveloppement \`a
l'origine d'une m\'etrique $g=\euc + H + \cdots $ sur $\setR^4$. La courbure en
$0$ de $g$ ne d\'epend que de $H$ et sera not\'ee $R(H)$. Il est important
de noter que $R(H)$ est \emph{lin\'eaire} en $H$. Les composantes de
$R(H)$ (courbure de Weyl, tenseur de Ricci et courbure scalaire)
seront not\'ees $W(H)$, $\Ric(H)$ et $\Scal(H)$. On d\'eduit aussi
l'op\'erateur de courbure,
$$ \bR(H) : \Lambda^2\setR^4 \to \Lambda^2\setR^4. $$
Dans la d\'ecomposition $\Lambda^2\setR^4=\Lambda^2_+\setR^4\oplus\Lambda^2_-\setR^4$, cet op\'erateur
s'\'ecrit
\begin{equation}
 \bR(H) =
\begin{pmatrix}
  \frac{\Scal(H)}{12}+W_+(H) & \Ric_0(H) \\
  \Ric_0(H) & \frac{\Scal(H)}{12}+W_-(H)
\end{pmatrix}.\label{eq:12}
\end{equation}

Enfin, il sera important dans la suite de noter que le tenseur $H$ n'est
pas forc\'ement donn\'e en jauge de Bianchi par rapport \`a la m\'etrique
euclidienne, et le d\'efaut \`a \^etre en jauge de Bianchi est mesur\'e par la
1-forme
$$ B_\euc H = \sum_k (-2 H_{kjkl}+H_{jlkk}) x^j dx^l . $$

Dans cette section, on \'etudie un probl\`eme infinit\'esimal de d\'eformation
de $\eh$ par un tenseur asymptotique \`a l'infini \`a un tenseur $H$ donn\'e
de type (\ref{eq:11}).

Supposons que le tenseur $H$ soit d'Einstein, c'est-\`a-dire
$$\Ric(H)=\Lambda\sum (dx^i)^2$$ pour un r\'eel $\Lambda$. Fixons les
coordonn\'ees $(x^i)$ sur $\setR^4$ comme dans la section
\ref{sec:la-linearisation-de}, avec la m\'etrique de Eguchi-Hanson
donn\'ee par la formule (\ref{eq:3}). Le probl\`eme que nous regardons
s'\'ecrit :
\begin{equation}
\begin{aligned}
  d_\eh \Ric(h) &= \Lambda \eh , \\
  h &= H + O(r^{-2+\epsilon}) \text{ pour tout }\epsilon>0.
\end{aligned}\label{eq:13}
\end{equation}

Le but de cette section est le r\'esultat suivant :
\begin{prop}\label{prop:cond-1}
  Le probl\`eme (\ref{eq:13}) a une solution si et seulement si, pour
  $i=1,2,3$, on a
$$
\lim_{r\to \infty} \int_{S_r/\setZ_2}
\left(\frac3r\langle H,o_i \rangle+o_i(BH,\frac\partial{\partial r})\right)
 \vol_{S_r/\setZ_2}=0,
$$
o\`u les $o_i$ sont les solutions $L^2$ de $P_\eh h=0$.
\end{prop}
Remarquons que les termes principaux de $H$, $BH$ et $o_i$ sont
respectivement en $r^2$, $r$ et $r^{-4}$, tandis que le volume des
sph\`eres est en $r^3$, donc la limite est bien d\'efinie. De m\^eme, le choix de
la m\'etrique euclidienne ou de la m\'etrique de Eguchi-Hanson pour calculer
$BH$ est sans importance, puisque seul le terme principal contribue \`a la
limite.
\begin{proof}
  Remarquons qu'on peut mettre $H$ en jauge de Bianchi, en
  trouvant un champ de vecteurs $V$ sur $\setR^4$, dont les
  coefficients sont des polyn\^omes homog\`enes de degr\'e 3, tel que
$$ B_\euc \tilde H=0, \quad \text{ o\`u }\quad\tilde H=H+\delta_\euc ^*V. $$
  En effet, il suffit de r\'esoudre l'\'equation $B_\euc \delta_\euc ^*V=-B_\euc H$,
or $B_\euc \delta_\euc =\nabla_\euc ^*\nabla_\euc =-\sum \frac{\partial^2}{(\partial x^i)^2}$ donc c'est toujours
possible, $\tilde H$ est encore un tenseur dont les coefficients sont
des polyn\^omes de degr\'e 2.

Si on a une solution du syst\`eme (\ref{eq:13}), alors on peut mettre $h$
en jauge de Bianchi en r\'esolvant
  \begin{equation}
    \label{eq:14}
    \begin{aligned}
      B_\eh \delta_\eh ^*v &= - B_\eh h , \\
      v &= V + O(r^{-1}) ,
    \end{aligned}
  \end{equation}
et en posant $\tilde h=h+\delta_\eh ^*v$, qui est alors solution du syst\`eme
\begin{equation}
  \label{eq:15}
  \begin{aligned}
  d_\eh \Ric(\tilde h) &= \Lambda \eh , \\
  B_\eh \tilde h&= 0, \\
  \tilde h &= \tilde H + O(r^{-2+\epsilon}),
\end{aligned}
\end{equation}
donc, en particulier,
\begin{equation}
  \label{eq:16}
  \begin{aligned}
    P_\eh \tilde h &= 0, \\
    \tilde h &= \tilde H + O(r^{-2+\varepsilon}).
  \end{aligned}
\end{equation}

R\'eciproquement, une solution de l'\'equation (\ref{eq:16}) satisfait
automatiquement le syst\`eme (\ref{eq:15}) : en effet, puisque
$B_\eh P_\eh =\frac12\nabla_\eh ^*\nabla_\eh  B_\eh $, si $P_\eh \tilde
h=0$, alors $w=B_\eh \tilde h$ satisfait $\nabla_\eh ^*\nabla_\eh  w=0$ ; en
outre, compte tenu de l'asymptotique de $\tilde h$, on a
$w=O(r^{-3+\epsilon})$, donc il faut que $\nabla_\eh  w=0$ ce qui entra\^\i ne $w=0$. On
a donc montr\'e que la r\'esolution du syst\`eme (\ref{eq:13}) se ram\`ene \`a la
r\'esolution du syst\`eme (\ref{eq:16}), o\`u $\tilde H$ est en jauge de
Bianchi \`a l'infini.

Or, la r\'esolution du syst\`eme (\ref{eq:16}) est maintenant facile : soit
$\chi$ une fonction de coupure, donc
$$ \chi(x)=
\begin{cases}
  1, &\quad |x|>2R, \\
  0, &\quad |x|<R,
\end{cases}$$
alors il faut trouver $k=\tilde h-\chi \tilde H$ satisfaisant
 \begin{equation}
   \label{eq:17}
   \begin{aligned}
 k&=O(r^{-2}),\\ P_\eh k&=-P_\eh (\chi\tilde H).
   \end{aligned}
 \end{equation}
Sachant que $\tilde H$ est en jauge de Bianchi \`a l'infini, et que
$\Ric(\tilde H)=\Lambda\euc$, on voit que
$$ B_\eh \tilde H=O(r^{-3}), \quad
   d_\eh \Ric(\tilde H)-\Lambda\eh=O(r^{-4}), $$
donc au total
$$ P_\eh (\chi\tilde H)=O(r^{-4}) . $$
Par la proposition \ref{prop:ker-P-EH},
une solution de (\ref{eq:17}) est donc possible si et seulement si
\begin{equation}
  \label{eq:18}
  \big(P_\eh (\chi\tilde H),o_i\big)=0, \quad i=1,2,3,
\end{equation}
o\`u les $o_i$ engendrent le noyau de $P_\eh $. Dans ce cas, la
solution $k$ satisfait $k=O(r^{-2+\epsilon})$ pour tout
$\epsilon>0$.

Puisque $P=\frac12 \nabla^*\nabla-\rR$, l'\'equation $P_\eh o_i=0$ entra\^\i ne
\begin{equation}
  \label{eq:19}
  \int_{r\leq R} \langle P_\eh v,o_i \rangle \vol_\eh 
 = \frac12 \int_{r=R} \big( \langle - \nabla_{\vec n}v,o_i \rangle
                          + \langle v,\nabla_{\vec n}o_i \rangle
                      \big) \vol_{S_R/\setZ_2} .
\end{equation}
Sachant que $o_i$ admet \`a l'infini un d\'eveloppement suivant les
puissances de $\frac1r$, dont le terme principal est en $r^{-4}$,
on d\'eduit 
$$ \nabla_{\vec n}o_i \sim \frac{-4}r o_i . $$
De m\^eme, \`a l'infini, $\tilde H$ est homog\`ene de degr\'e 2, donc
$$ \nabla_{\vec n}\tilde H \sim \frac 2r \tilde H . $$
Passant \`a la limite dans (\ref{eq:19}), on d\'eduit
\begin{equation}
  \label{eq:20}
  \big(P_\eh (\chi\tilde H),o_i\big)
 = -3 \lim_{r\to\infty}
 \frac1r\int_{S_r/\setZ_2} \langle \tilde H,o_i \rangle \vol_{S_R/\setZ_2}.
\end{equation}

Il reste \`a exprimer cette int\'egrale en fonction de $H$ plut\^ot que $\tilde
H$. Pour cela nous modifions un peu le calcul pr\'ec\'edent : rappelons que
$\tilde H=H+\delta_\euc ^*V$, o\`u $V$ a \'et\'e choisi de sorte que
$B_\euc \delta_\euc ^*V=-B_\euc H$. Prolongeons $V$ de mani\`ere quelconque sur
tout $T^*P_\setC^1$, en pr\'eservant son asymptotique, alors $\chi\tilde
H=H'+\delta_\eh ^*V$, avec $H'=H+O(r^{-2})$.  La m\^eme int\'egration par parties
que pr\'ec\'edemment m\`ene \`a :
\begin{equation}
  \label{eq:21}
  (P_\eh H',o_i) = -3 \lim_{r\to\infty}
 \frac1r\int_{S_r/\setZ_2} \langle H,o_i \rangle \vol_{S_R/\setZ_2}.
\end{equation}
Faisons en revanche une int\'egration par parties diff\'erente sur l'autre
terme : puisque
$P_\eh \delta_\eh ^*=\delta_\eh ^*B_\eh \delta_\eh ^*$
et $B_\eh o_i=0$, on obtient
$$
\int_{r\leq R} \langle P_\eh \delta_\eh ^*V,o_i\rangle \vol_\eh 
= \int_{S_R/\setZ_2} *\big(B_\eh \delta_\eh ^*V\lrcorner o_i \big).
$$
Compte tenu de $B_\euc \delta_\euc ^*V=-B_\euc H$ pr\`es de l'infini,
en passant \`a la limite $R\to\infty$, on obtient
\begin{equation}
  \label{eq:22}
  (P_\eh \delta_\eh ^*V,o_i) =
 \lim_{r\to\infty} \int_{S_r/\setZ_2} *\big(-BH\lrcorner o_i \big).
\end{equation}
La somme de (\ref{eq:21}) et (\ref{eq:22}) donne le r\'esultat de la proposition.
\end{proof}

Si la condition de la proposition \ref{prop:cond-1} n'est pas satisfaite,
on ne peut plus r\'esoudre le syst\`eme (\ref{eq:13}). N\'eanmoins, on peut
r\'esoudre modulo l'espace d'obstruction $\cO_\eh$, c'est-\`a-dire r\'esoudre le
syst\`eme suivant :
\begin{equation}
  \label{eq:23}
  \begin{aligned}
    d_\eh \Ric(h) &= \Lambda \eh + \sum_1^3 \lambda_io_i , \\
    h &= H + O(r^{-2+\epsilon}) \text{ pour tout }\epsilon>0,
  \end{aligned}
\end{equation}
o\`u les $\lambda_i$ sont des constantes inconnues \`a trouver.

\begin{prop}\label{pro:sol-EH-O}
  Le probl\`eme (\ref{eq:23}) a toujours une solution $h$. Les
  constantes $\lambda_i$ sont d\'etermin\'ees par $H$ :
  \begin{equation}
   \lambda_i = - \frac1{\|\Omega\|_{L^2}^2}\lim_{r\to \infty} \int_{S_r/\setZ_2}
  \left(\frac3r\langle H,o_i \rangle+o_i(BH,\frac\partial{\partial r})\right) \vol_{S_r/\setZ_2}.\label{eq:24}
\end{equation}
  La solution est d\'etermin\'ee modulo jauge et modulo le noyau $\cO_\eh$. En
  mettant $H$ en jauge de Bianchi, on peut fixer une solution particuli\`ere en
  demandant
  \begin{equation}
   B_\eh h=0, \quad \int_{S^2} \langle h,o_j\rangle_\eh \vol = 0.\label{eq:25}
  \end{equation}
\end{prop}
\begin{proof}
  Observons que les tenseurs $k_i$ de la proposition
  \ref{prop:ker-P-EH} n'ont pas de raison particuli\`ere d'\^etre en jauge
  de Bianchi, mais on peut les mettre en jauge de Bianchi en posant
  $\tilde k_i=k_i+\delta_\eh ^*v_i$, o\`u $v_i$ r\'esoud le syst\`eme
  (\ref{eq:14}).

  On montre alors la proposition en reprenant la d\'emonstration de la
  proposition \ref{prop:cond-1} : il s'agit de r\'esoudre le syst\`eme
  (\ref{eq:17}) pour un $\tilde H$ en jauge de Bianchi \`a l'infini. La
  solution est possible pourvu que les trois conditions (\ref{eq:18})
  soient satisfaites. En posant $\tilde H=\tilde H'+\sum_1^3\lambda_i\tilde
  k_i$, o\`u les $\lambda_i$ sont donn\'ees par (\ref{eq:24}), alors $\tilde H'$
  satisfait les conditions (\ref{eq:18}) et on peut donc r\'esoudre le
  probl\`eme. Le prix \`a payer est que $P_\eh \tilde k_i=d_\eh \Ric\tilde
  k_i=o_i$, donc on n'a r\'esolu que l'\'equation (\ref{eq:23}).
\end{proof}

\section{Reformulation des obstructions}
\label{sec:reform-des-obstr}

Le but de cette section est de recalculer la condition de r\'esolution
du syst\`eme (\ref{eq:13}) sous une forme diff\'erente. En effet, la
condition sur $H$ trouv\'ee dans la proposition \ref{prop:cond-1} doit
manifestement \^etre invariante par diff\'eomorphismes infinit\'esimaux de
$\setR^4$ dont la diff\'erentielle est l'identit\'e \`a l'origine : donc cette
condition ne peut \^etre qu'une condition sur la courbure $R(H)$
associ\'ee \`a $H$. En effet, on obtient une condition simple sur
l'op\'erateur de courbure $\bR(H)$ agissant sur $\Omega^2\setR^4$ :
\begin{theo}\label{theo:cond-2}
  Le syst\`eme (\ref{eq:13}) a une solution si et seulement si
  \begin{equation}
  \bR^+(H)(dx^1\land dx^2+dx^3\land dx^4)=0 .\label{eq:26}
  \end{equation}
\end{theo}
Notons qu'on pourrait \'ecrire aussi bien la condition sur $\bR(H)$, puisque
la m\'etrique est d'Einstein.

Dans ce th\'eor\`eme, les coordonn\'ees $(x^i)$ ne sont pas quelconques, puisqu'\`a
l'infini $x^1+ix^2$ et $x^3+ix^4$ sont complexes pour la structure complexe
$I_1$ de la m\'etrique de Eguchi-Hanson---\`a savoir la structure de $T^*
P^1_\setC$.

Pour d\'emontrer le th\'eor\`eme, on commence par calculer explicitement les
obstructions provenant de la proposition \ref{prop:cond-1}.
Compte tenu de la proposition \ref{prop:ker-P-EH} de la formule
(\ref{eq:5}), on a les asymptotiques explicites suivantes pour les
$o_i$ :
\begin{equation}
    \label{eq:27}
    \begin{aligned}
      o_1 &\sim \frac1{r^2} \big( \frac{dr^2}{r^2} + \alpha_1^2 - \alpha_2^2 - \alpha_3^2 \big),\\
      o_2 &\sim \frac1{r^2} \big( \frac{dr}r \cdot \alpha_3 + \alpha_1 \cdot \alpha_2 \big),\\
      o_3 &\sim \frac1{r^2} \big(-\frac{dr}r \cdot \alpha_2 + \alpha_1 \cdot \alpha_3 \big).
    \end{aligned}
\end{equation}
Manifestement, la condition d'existence d'une solution au syst\`eme
(\ref{eq:13}) est invariante sous la transformation $H\to H+\delta^*V$, o\`u $V$
est un champ de vecteurs dont les coefficients sont homog\`enes de degr\'e
3. Par cons\'equent, il suffit de tester la condition de la proposition
\ref{prop:cond-1} sur $H$ mis en jauge radiale : $\frac \partial{\partial r}\lrcorner
H=0$. Cela implique que $H$ s'\'ecrive sous la forme
\begin{equation}
   \label{eq:28}
 H = \sum_{ij} H_{ij}\beta_i\beta_j, 
\end{equation}
o\`u les $H_{ij}$ sont des fonctions homog\`enes de degr\'e 0, et on a pos\'e
$$ \beta_i = r^2\alpha_i = x^1 dx^{i+1} - x^{i+1} dx^1 + \cdots $$

On calcule alors :
\begin{lemm}\label{lemm:calcul-obstructions}
  On a 
  \begin{multline}
    \label{eq:29}
    \lim_{r\to \infty} \int_{S_r/\setZ_2} \left(\frac3r\langle H,o_1 \rangle+o_1(BH,\frac\partial{\partial r})\right)
 \vol_{S_r/\setZ_2} \\ = \int_{S^3/\setZ_2}(5H_{11}-H_{22}-H_{33}) \vol_{S^3/\setZ_2} ,
 \end{multline}
et pour $i=2,3$,
\begin{equation}
  \label{eq:30}
  \lim_{r\to \infty} \int_{S_r/\setZ_2} \left(\frac3r\langle H,o_i \rangle+o_i(BH,\frac\partial{\partial r})\right)
 \vol_{S_r/\setZ_2}= \int_{S^3/\setZ_2} 6 H_{1i}\vol_{S^3/\setZ_2} .
\end{equation}
\end{lemm}
\begin{proof}
  On a $\nabla\beta_1=dx^1\otimes dx^2-dx^2\otimes dx^1+dx^3\otimes
  dx^4-dx^4\otimes dx^3=I_1$, et, plus g\'en\'eralement,
$$ \nabla\beta_i = I_i . $$
  On calcule alors ais\'ement, toujours pour la m\'etrique euclidienne,
  \begin{align*}
    B\beta_i^2 &= 2r dr \\
    B(\beta_i\beta_j+\beta_j\beta_i) &= -2 \epsilon_{ij}^k \beta_k .
  \end{align*}
  Enfin, on aura besoin, pour une fonction $f$, de la formule :
$$ B(fT) = f BT + \big(-\nabla f \lrcorner T + \frac12(\tr T)df\big) . $$

  Remarquons \`a pr\'esent que $H=H_{ij}\beta_i\beta_j$ avec $H_{ij}$ homog\`ene de
  degr\'e 0, donc $\frac{\partial H_{ij}}{\partial r}=0$. On obtient alors imm\'ediatement
$$ o_1(BH,\frac \partial{\partial r})=BH(\frac \partial{\partial r})=2(H_{11}+H_{22}+H_{33}); $$
  d'un autre c\^ot\'e,
$$\langle H,o_1\rangle=H_{11}-H_{22}-H_{33} ;$$
ces deux formules ensemble donnent la premi\`ere formule du lemme, les
deux autres s'\'etablissent de la m\^eme mani\`ere.
\end{proof}

\begin{proof}[D\'emonstration du th\'eor\`eme \ref{theo:cond-2}]
  Il reste \`a interpr\'eter l'annulation de ces trois quantit\'es en termes
de la courbure $R(H)$. Bien qu'un calcul explicite soit possible, nous
pr\'ef\'erons l'approche abstraite suivante. Notons $S_+$ et $S_-$ les
repr\'esentations spinorielles de $\Spin(4)=\SU(2)\times \SU(2)$, ce sont donc
les repr\'esentations standard dans $\setC^2$ des deux copies de $\SU(2)$ ;
notons $S_\pm^k$ leur produit sym\'etrique, alors les repr\'esentations de
$\SO(4)$ se d\'ecomposent en somme de produits tensoriels
$S_+^kS_-^l$. En particulier, $\setR^4=S_+S_-$ (en fait c'est la partie
r\'eelle, mais comme toutes nos repr\'esentations sont r\'eelles, nous
oublierons ce d\'etail dans la suite), les 2-formes se d\'ecomposent en
$\Omega^2\setR^4=\Omega^2_+\oplus\Omega^2_-=S_+^2\oplus S_-^2$, et les 2-tenseurs sym\'etriques en
$S^2\setR^4=S_+^2S_-^2\oplus\setR$. Finalement, les deux demi-tenseurs de Weyl
$W_\pm$ vivent dans les repr\'esentations $S_\pm^4$.

Le 2-tenseur $H=H_{ijkl}x^ix^jdx^kdx^l$ vit dans la repr\'esentation
$$\cH=S^2\setR^4\otimes S^2\setR^4,$$ mais la condition qu'il soit radial se traduit
par $$V(H):=x^i\frac \partial{\partial x^i}\lrcorner H=0 .$$ Or les coefficients du
vecteur $V(H)$ sont homog\`enes de degr\'e 3, donc $V(H)\in \cV=S^3\setR^4\otimes\setR^4$.
On a les d\'ecompositions :
\begin{align}
\label{eq:31}
  \cH&=S_+^4S_-^4\oplus S_+^4S_-^2\oplus S_+^4\oplus S_+^2S_-^4\oplus3S_+^2S_-^2\oplus S_+^2\oplus S_-^4\oplus S_-^2\oplus2\setR\\
\label{eq:32}
  \cV&=S_+^4S_-^4\oplus S_+^4S_-^2\oplus S_+^2S_-^4\oplus2S_+^2S_-^2\oplus S_+^2\oplus S_-^2\oplus\setR.
\end{align}
Il est clair que $V:\cH\to\cV$ est surjectif. Comme il est aussi
$\SO(4)$-\'equivariant, on en d\'eduit que l'espace des 2-tenseurs $H$
radiaux, \`a savoir $\ker V$, se d\'ecompose suivant les repr\'esentations
pr\'esentes dans (\ref{eq:31}) mais absentes dans (\ref{eq:32}), donc
\begin{equation}
  \label{eq:33}
  \ker V =  \setR \oplus S_+^2S_-^2 \oplus S_+^4 \oplus S_-^4 .
\end{equation}
Cette d\'ecomposition co\"\i ncide avec celle de l'ensemble $\cR$ des
tenseurs de courbure \`a l'origine de $\setR^4$, qui, au vu des diff\'erentes
composantes irr\'eductibles de la courbure, voir (\ref{eq:12}), s'\'ecrit aussi
\begin{equation}
  \label{eq:34}
  \cR = \setR \oplus S_+^2S_-^2 \oplus S_+^4 \oplus S_-^4 .
\end{equation}

Cette co\"\i ncidence traduit le fait bien connu de g\'eom\'etrie riemannienne
suivant lequel, dans des coordonn\'ees g\'eod\'esiques, les d\'eriv\'ees
secondes de la m\'etrique \`a l'origine sont d\'etermin\'ees par la courbure \`a
l'origine.  En cons\'equence, la courbure
\begin{equation}
  \label{eq:35}
  R : \ker V \longrightarrow \cR ,
\end{equation}
est un isomorphisme. Compte tenu de la d\'ecomposition en facteurs
irr\'eductibles, cet isomorphisme est compl\`etement d\'etermin\'e par la
d\'ecomposition, \`a une constante pr\`es dans chaque facteur.

Nous pouvons donc expliciter l'isomorphisme (\ref{eq:35}) en calculant
un exemple pour fixer les constantes. Dans notre cas, il
suffira de traiter les facteurs $\setR$ et $S_+^4$. Utilisons la
description suivante des \'el\'ements de $\setR\oplus S_+^4$ :
\begin{itemize}
\item dans $\ker V$, il s'agit des tenseurs qui se laissent \'ecrire
  $H=\sum H_{ij}\beta_i\beta_j$, o\`u les $H_{ij}$ sont
  \emph{constantes} ;
\item dans $\cR$, on peut identifier, via l'op\'erateur de courbure, le
  facteur $\setR\oplus S_+^4$ aux endomorphismes sym\'etriques de
$\Omega_+^2$ : on les \'ecrira donc comme des matrices dans la base de
$\Omega_+^2$ donn\'ee par les
$\omega_i(\cdot,\cdot)=g(I_i\cdot,\cdot)=d\beta_i$. 
\end{itemize}

Passons maintenant au calcul d'un exemple :
la m\'etrique hyperbolique complexe, \`a courbure sectionnelle
  holomorphe $-1$, s'\'ecrit
$$g=dr^2+\sinh^2(r)\alpha_1^2+4\sinh^2(\frac r2)(\alpha_2^2+\alpha_3^2),$$
ce qui correspond \`a
 \begin{equation}
   \label{eq:36}
    H = \frac13\big( \beta_1^2 + \frac{\beta_2^2+\beta_3^2}4 \big) ; 
 \end{equation}
on a $\Scal=-6$ et, comme pour toute m\'etrique $I_1$-k\"ahl\'erienne,
l'op\'erateur de courbure, restreint \`a $\Omega_+^2$ est donn\'e par
\begin{equation}
  \label{eq:37}
   \bR^+=\frac{\Scal}{12} +
\begin{pmatrix}
  \frac{\Scal}6 & & \\ & -\frac{\Scal}{12} & \\ & & -\frac{\Scal}{12}
\end{pmatrix}.
\end{equation}
La comparaison de (\ref{eq:36}) et (\ref{eq:37}) permet de fixer les
constantes sur les facteurs $\setR$ et $S_+^4$, on obtient
\begin{equation}
  \label{eq:38}
  \bR^+(\sum H_{ij}\alpha_i\alpha_j)=
  \left(\begin{smallmatrix}
    -5 H_{11}+H_{22}+H_{33} & - 6 H_{12} & - 6 H_{13} \\
    -6 H_{21} & -5 H_{22}+H_{33}+H_{11} & -6 H_{23} \\
    -6 H_{31} & -6 H_{32} & -5 H_{33} + H_{11} + H_{22}
  \end{smallmatrix}\right).
\end{equation}
Via le lemme \ref{lemm:calcul-obstructions}, il appara\^\i t ainsi que les
annulations requises sont \'equivalentes \`a $\bR^+(H)(I_1)=0$. Le th\'eor\`eme est
d\'emontr\'e.
\end{proof}

\begin{rema}
  La formule (\ref{eq:38}), avec le lemme
  \ref{lemm:calcul-obstructions}, donne une formule directe pour les
  coefficients $\lambda_i$ de la proposition \ref{pro:sol-EH-O} en termes de
  $\bR(H)(I_1)$.
\end{rema}

\section{Construction d'une solution approch\'ee}
\label{sec:constr-dune-solut}

Soit $(M_0,g_0)$ une vari\'et\'e d'Einstein, donc
$\Ric_{g_0}=\Lambda g_0$. Supposons que $M_0$ ait une singularit\'e orbifold de
type $\setR^4/\setZ_2$ au point $p_0$. On va construire sur la
d\'esingularisation $M$ de la vari\'et\'e une famille de m\'etriques
$(g_t)_{t>0}$, satisfaisant approximativement l'\'equation d'Einstein,
modulo les obstructions pr\'esentes sur Eguchi-Hanson. Jusqu'\`a la section
\ref{sec:resol-mod-O} incluse, la vari\'et\'e $M_0$ peut \^etre aussi bien
compacte que non compacte.

Dans des coordonn\'ees g\'eod\'esiques $(x^i)$ au point $p_0$, on \'ecrit
\begin{equation}
 g_0 = \euc + H + O(|x|^4), \text{ avec }
H = H_{ijkl}x^ix^jdx^kdx^l.\label{eq:39}
\end{equation}
(Le d\'eveloppement ne contient que des termes pairs, au vu de l'invariance
sous $\setZ_2$). Par la proposition \ref{pro:sol-EH-O}, il existe sur
Eguchi-Hanson une solution $h$ au syst\`eme (\ref{eq:23}), c'est-\`a-dire une
solution de $d_\eh \Ric(h)=\Lambda\eh+\sum_1^3\lambda_io_i$, asymptote \`a $H$ \`a l'infini,
o\`u les constantes $\lambda_i$ sont d\'etermin\'ees par $H$. On prendra $h$ normalis\'ee
par (\ref{eq:25}).

Remarquons la pr\'esence d'un param\`etre de jauge dans la construction : si
$\varphi\in \SO(4)$, alors on peut appliquer $\varphi$ \`a la carte dans laquelle est
exprim\'ee $g_0$, ce qui a pour effet de modifier $H$ dans (\ref{eq:39}), et
donc la solution $h$. De mani\`ere \'equivalente, on peut penser qu'on a
appliqu\'e $\varphi^{-1}$ \`a la m\'etrique de Eguchi-Hanson. Comme celle-ci est
$\U(2)$-invariante, il ne reste en r\'ealit\'e que deux param\`etres r\'eels, dans
$\SO(4)/\U(2)$. Donc on ne perd rien \`a restreindre $\varphi$ \`a \^etre dans le groupe
$\Sp_1\subset \SO(4)$ qui fixe $\Omega^2_-\setR^4$, et le vrai param\`etre est dans
$\Sp_1/U_1=\SO(4)/\U(2)$.

Notons que ce param\`etre de jauge peut \^etre d\'ej\`a utilis\'e de la mani\`ere
suivante : dans le cas o\`u $\bR^+_{g_0}(p_0)|_{\Omega_+^2\setR^4}$ a un noyau non
trivial, quitte \`a appliquer une jauge $\varphi\in \Sp_1/U_1$, on peut supposer que ce
noyau contient la forme $dx^1\land dx^2+dx^3\land dx^4$. Alors, par le th\'eor\`eme
\ref{theo:cond-2}, on a $\lambda_j=0$. Plus g\'en\'eralement, toujours gr\^ace \`a ce
param\`etre de jauge, on peut assurer que $dx^1\land dx^2+dx^3\land dx^4$ soit un
vecteur propre de $\bR(H)$, ce qui implique au moins $$\lambda_2=\lambda_3=0.$$

Consid\'erons alors sur Eguchi-Hanson le 2-tenseur
$$ h_{t,\varphi} = \eh + t h . $$
La notation ici garde en m\'emoire le param\`etre de jauge $\varphi\in \Sp_1/U_1$.
A priori il ne s'agit pas d'une m\'etrique, puisque le terme $th$
explose \`a l'infini et n'a pas de raison d'\^etre positif. N\'eanmoins,
$h_{t,\varphi}$ est une m\'etrique sur des r\'egions de plus en plus larges de
Eguchi-Hanson. Pour l'\'enoncer correctement, voyons Eguchi-Hanson dans
les coordonn\'ees de la section \ref{sec:la-linearisation-de}, donc
identifi\'ee en dehors de la section nulle avec $\setR^4/\setZ_2$, et
d\'efinissons des r\'egions de Eguchi-Hanson par
\begin{align*}
  \EH^t_> &= \{ \frac12 t^{-\beta}\leq r\leq 2t^{-\beta}\} ,\\
  \EH^t &= \{ r\leq 2t^{-\beta} \},
\end{align*}
o\`u $\beta\in ]0,\frac12[$ sera fix\'e ult\'erieurement. D\'efinissons sur
Eguchi-Hanson une fonction $\rho\geq 1$ asymptote au rayon $r$ pr\`es de
l'infini : $\rho$ est une fonction croissante de $r$ satisfaisant
$$
\rho =
\begin{cases}
  r & \text{ si }r>2, \\ 1 & \text{ si }r\leq 1.
\end{cases}
$$
\begin{lemm}\label{lem:est1-Ric}
  Soit $\beta\in ]0,\frac12[$, si $t\ll 1$ alors $h_{t,\varphi}$ est une m\'etrique sur
  la r\'egion $\EH^t$. En outre, sur cette r\'egion, 
$$ \big|\nabla^k\big(\Ric_{h_{t,\varphi}}-t\Lambda h_{t,\varphi}-t\sum_1^3\lambda_jo_j\big)\big|_\eh  \leq c_k t^2 \rho^{2-k}. $$
\end{lemm}
\begin{proof}
  Puisque $h=O(\rho^2)$, sur $\EH^t$ on a $|h|\leq ct^{1-2\beta}$ ; or $\beta<\frac12$,
  d'o\`u il r\'esulte que $h_{t,\varphi}$ est une m\'etrique d\`es que $t$ est assez
  petit.

  Comme $h$ est solution de $d_\eh \Ric(h)=\Lambda\eh+\sum_1^3\lambda_io_i$, le
  membre de gauche de l'in\'egalit\'e est constitu\'e de $ - t^2 h$ et de termes
  au moins quadratiques de $\Ric$ appliqu\'es \`a $th$ : comme $th$ est petit,
  les termes les plus grands sont les termes quadratiques, qui
  contiennent toujours au moins 2 d\'eriv\'ees de $h$ (termes en $\nabla^2h \cdot h$, ou
  $\nabla h \cdot \nabla h$) ; vu que $|\nabla^kh|\leq c_k\rho^{2-k}$, on en d\'eduit le r\'esultat.
\end{proof}
Consid\'erons l'homoth\'etie $s_t:\EH^{t,\beta}_>\to M_0$ donn\'ee par
$s_t(x)=\sqrt{t}x$. Alors $\frac1t \varphi^*s_t^*g_0$ est proche de $h_{t,\varphi}$ :
sur $\EH^{t,\beta}_>$,
\begin{align*}
  \frac1t \varphi^*s_t^*g_0 &= \euc + tH + O(t^2r^4) , \\
  h_{t,\varphi} &= \euc + tH + O(r^{-4}) + O(tr^{-2+\varepsilon}) ,
\end{align*}
et la d\'eriv\'ee $k$-i\`eme satisfait la m\^eme estimation, avec les erreurs
multipli\'ees par $r^{-k}$.
\emph{Une co\"\i ncidence optimale est obtenue avec le
choix $$\beta=\frac14,$$ qui sera dor\'enavant fix\'e.}
Par cons\'equent, sur $\EH^t_>$,
$$ |\nabla^k(\frac1t \varphi^*s_t^*g_0 - h_{t,\varphi})|_\euc  \leq c_k t^{2+\frac k4}. $$ 
Choisissons une fonction de coupure $\chi_t=\chi_t(r)$, \`a support dans
$\EH^t$, telle que
\begin{equation}\label{eq:40}
 \chi_t(r) = 1 \text{ si }r\leq \frac12 t^{-\frac14},
 \quad |\partial_r^k\chi_t|\leq c_k t^{\frac14 k}.
\end{equation}
Il suffit de choisir $\chi_t(\rho)=\chi_1(t^{\frac14}\rho)$.
On construit une m\'etrique $g_{t,\varphi}$ en recollant :
\begin{itemize}
\item la m\'etrique $g_0$ de $M_0$ d\`es que $r\geq 2t^{\frac14}$,
\item la m\'etrique $t[(1-\chi_t)\frac{s_t^*g_0}t+\chi_t\varphi_*h_{t,\varphi}]$ sur $\EH^t$,
\end{itemize}
via l'homoth\'etie $s_t$. Les deux m\'etriques co\"\i ncident au voisinage de
$r=2t^{\frac 14}$, donc on a bien d\'efini une m\'etrique sur la
d\'esingularisation $M$ de $M_0$, obtenue comme
$$ M = M_0^t \cup_{s_t} \EH^t , \text{ avec } M_0^t=\{r>\frac12
t^{\frac14}\}\cap M_0 . $$

On montre de mani\`ere similaire au lemme \ref{lem:est1-Ric} :
\begin{lemm}\label{lem:est-ini-Ric}
  La m\'etrique $g_{t,\varphi}$ satisfait l'estimation sur $\EH^t$ :
$$ \big|\nabla^k(\Ric_{g_{t,\varphi}}-\Lambda g_{t,\varphi}-t\sum_1^3\lambda_j\chi_to_j)\big|_\eh  \leq c_k t^2 \rho^{2-k}. $$\qed
\end{lemm}
On prendra garde que, dans ce lemme, le membre de gauche est
concentr\'e sur $\EH^t$, et la norme est exprim\'ee par rapport \`a la
m\'etrique de Eguchi-Hanson, comme l'indique la notation. Par ailleurs, comme
on a vu, les $\lambda_j$ d\'ependent aussi de $\varphi$.

On remarquera \'egalement que l'erreur sur $o_j$ provoqu\'ee par la
troncature est absorb\'ee dans le terme d'erreur $t^2\rho^2$.

\begin{rema}
  Un choix plus simple de recollement serait la m\'etrique
  \begin{equation}
    \label{eq:41}
    g_{t,\varphi}^0 = \chi_t t\eh + (1-\chi_t)\varphi^*g_0,
  \end{equation}
  dans laquelle on n'a pas fait co\"\i ncider les termes quadratiques de
  $g_0$ en $p_0$ avec ceux d'une m\'etrique sur $\EH$. Mais on v\'erifie
  imm\'ediatement que cette m\'etrique est une moins bonne approximation
  d'une m\'etrique d'Einstein, car elle v\'erifie seulement sur $\EH^t$
  \begin{equation}
    \label{eq:42}
    \big|\nabla^k(\Ric_{g_{t,\varphi}^0}-\Lambda g_{t,\varphi}^0)\big|_\eh \leq c t \rho^{-k},
  \end{equation}
  qui est moins bon que l'estimation du lemme \ref{lem:est-ini-Ric}.
\end{rema}

\begin{nota}
  Pour all\'eger les notations dans la suite, on oubliera g\'en\'eralement
  la d\'ependance en $\varphi$ pour noter $g_t$ au lieu de $g_{t,\varphi}$. La d\'ependance
  en $\varphi$ sera \'ecrite quand elle est utilis\'ee.
\end{nota}

\section{\'Equation d'Einstein modulo les obstructions}
\label{sec:Einstein-mod-O}

On a construit dans la section \ref{sec:constr-dune-solut} une
solution approch\'ee $g_t$ de l'\'equation d'Einstein sur $M$. On souhaite
d\'eformer $g_t$ en une solution exacte, en appliquant la m\'ethode de
\cite[chapitre I]{Biq00} : pour chaque $t>0$ assez petit, on cherche
une m\'etrique $g$ proche de $g_t$, telle que
\begin{equation}
 \Phi_{g_t}(g) := \Ric_g-\Lambda g+\delta_g^*B_{g_t}g = 0 .\label{eq:43}
\end{equation}
En effet, par l'identit\'e de Bianchi, une solution $g$ doit satisfaire
$$ B_g\delta_g^*B_{g_t}g=0 . $$
Souvent \cite[lemme I.1.4]{Biq00}, l'op\'erateur $B_g\delta_g^*$ est
injectif, donc en r\'ealit\'e $B_{g_t}g=0$ et donc $\Ric_g=\Lambda g$.
Le lemme \ref{lem:Bdelta-inv} donnera des conditions sous lesquelles
$B_g\delta_g^*$ est inversible.

Or, l'application de cette m\'ethode n\'ecessite un contr\^ole de la norme
de l'inverse de la lin\'earisation de $\Phi$. Comme $P_\eh $ n'est pas
inversible pour la m\'etrique de Eguchi-Hanson, un tel contr\^ole n'existe
pas. Le conoyau $\cO_\eh$ de $P_\eh $, \'etudi\'e dans la section
\ref{sec:la-linearisation-de}, appara\^\i t ainsi comme obstruction \`a la
r\'esolution de l'\'equation d'Einstein. Aussi, dans un premier temps,
r\'esoudra-t-on seulement \guillemotleft{} modulo les obstructions \guillemotright.

Par la proposition \ref{prop:ker-P-EH}, l'espace $\cO_\eh $ est engendr\'e
par les trois tenseurs $o_i$. Pour poser l'\'equation d'Einstein modulo ce
conoyau, on peut d\'efinir un espace d'obstructions tronqu\'ees, de dimension
trois, $\cO=\langle\chi_to_i\rangle$, et \'ecrire
\begin{equation*}
 \Ric_g - \Lambda g \in \cO.
\end{equation*}
N\'eanmoins, cette \'equation pose probl\`eme, car le membre de gauche
satisfait l'identit\'e de Bianchi $B_g(\Ric_g-\Lambda g)=0$, qu'un \'el\'ement de
$\cO$ n'a pas de raison de satisfaire a priori. Il semble ainsi
difficile de r\'esoudre l'\'equation en appliquant la m\'ethode de
l'\'equation (\ref{eq:43}).

On contourne cette difficult\'e en faisant d\'ependre le second membre de
l'\'equation de la m\'etrique inconnue $g$. Pour toute m\'etrique $g$, telle
que l'op\'erateur $B_g\delta_g^*$ sur les champs de vecteurs soit un
isomorphisme (voir le lemme \ref{lem:Bdelta-inv}), on d\'efinit la
projection
\begin{equation}
\pi_g:L^2(\Omega^1)\to L^2(\Omega^1)\cap \ker B_g,\label{eq:44}
\end{equation}
par
\begin{equation}
  \label{eq:45}
  \pi_g = 1 - \delta_g^*(B_g\delta_g^*)^{-1}B_g.
\end{equation}
Faisons d\'ependre l'espace d'obstruction de la m\'etrique $g$ en
posant
\begin{equation}
  \label{eq:46}
  \cO_g = \langle \pi_g(\chi_to_i) \rangle.
\end{equation}
L'\'equation \`a r\'esoudre devient
\begin{equation}
  \label{eq:47}
  \Ric_g - \Lambda g \in \cO_g.
\end{equation}
\`A pr\'esent, la m\'ethode de l'\'equation (\ref{eq:43}) peut \^etre appliqu\'ee :
en int\'egrant la jauge dans l'\'equation, on r\'esout
\begin{equation}
  \label{eq:48}
   \Phi_{g_t}(g) = \Ric_g-\Lambda g+\delta_g^*B_{g_t}g \in \cO_g.
\end{equation}
Puisque $\cO_g\subset \ker B_g$, la m\^eme m\'ethode que ci-dessus assure, si
$B_g\delta_g^*$ est injectif, qu'une solution de $\Phi_{g_t}g\in \cO_g$
satisfait en r\'ealit\'e l'\'equation de jauge $B_{g_t}g=0$. On a donc montr\'e :
\begin{lemm}\label{lem:prob-mod-O}
  Si la m\'etrique $g$ satisfait $\Phi_{g_t}g\in \cO_g$, et si l'op\'erateur
  $B_g\delta_g^*$ est injectif, alors $B_{g_t}g=0$ et $\Ric_g-\Lambda g\in \cO_g$.\qed
\end{lemm}

\section{Espaces fonctionnels}
\label{sec:espaces-fonctionnels}

Construisons maintenant les espaces fonctionnels n\'ecessaires \`a
la r\'esolution du probl\`eme (\ref{eq:47}).

Fixons un rayon $r_{p_0}$ inf\'erieur au rayon d'injectivit\'e de $g_0$ au
point $p_0$. Soit $\tilde r$ une fonction r\'eelle d\'efinie sur $M_0$ par
$$
\tilde r =
\begin{cases}
  r & \text{ si }r\leq \frac12 r_{p_0}, \\
  1 & \text{ sur }M_0-B(p_0,r_{p_0}).
\end{cases}
$$

Par ailleurs, si $M_0$ est AH, on dispose d'un bout
$N=]0,x_0[\times\partial M$ de $M_0$, o\`u la premi\`ere coordonn\'ee, $x$, est une
\'equation du bord \`a l'infini. On prolonge $x$ sur $M_0$ de sorte que
$x_0\leq x\leq 1$ sur $M_0-N$.

Soit un fibr\'e tensoriel $E$, et $\delta_0\in \setR$ un poids. On d\'efinit sur
$\EH$ un espace \`a poids $C^{k,\alpha}_{\delta_0}(\EH,E)$ de sections de $E$ par :
\begin{align*}
  |s|_\alpha &= \sup_{d(x,y)<i_\eh } \frac{|s(x)-s(y)| }{d(x,y)^\alpha},\\
  \|s\|_{C^{k,\alpha}_{\delta_0}} &= \big(\sum_0^k \sup \rho^{\delta_0+k}|\nabla^ks|\big) +
  |\rho^{\delta_0+k+\alpha}\nabla^ks|_\alpha . \\
\end{align*}
Ici $i_\eh $ est le rayon d'injectivit\'e de la m\'etrique de
Eguchi-Hanson, et la diff\'erence $|s(x)-s(y)|$ est \'evalu\'ee, de mani\`ere
usuelle, en transportant parall\`element $s(y)$ en $E_x$ le long de la
g\'eod\'esique de $x$ \`a $y$.

De mani\`ere similaire, on d\'efinit un espace \`a poids
$C^{k,\alpha}_{\delta_0}(M_0,E)$ par 
\begin{align}
  \label{eq:49}
  \|s\|_{C^{k,\alpha}_{\delta_0}(M_0)} &= \big(\sum_0^k \sup \tilde r^{\delta_0+k}|\nabla^ks|\big) +
  |\tilde r^{\delta_0+k+\alpha}\nabla^ks|_\alpha ,\\
\intertext{puis, pour un couple de poids $(\delta_0,\delta_\infty)$, un espace \`a poids
$C^{k,\alpha}_{\delta_0,\delta_\infty}(M_0,E)$ par}
  \label{eq:50}
    \|s\|_{C^{k,\alpha}_{\delta_0,\delta_\infty}(M_0)} &=  \|x^{-\delta_\infty}s\|_{C^{k,\alpha}_{\delta_0}(M_0)} .
\end{align}
(Dans le cas o\`u $M_0$ est compact, on n'a plus besoin du poids $x$ \`a
l'infini, on peut garder la m\^eme d\'efinition en convenant que $x\equiv1$, donc le
poids $\delta_\infty$ ne joue plus de r\^ole).

Finalement la norme effectivement utilis\'ee dans cet article est un
recollement des normes sur $\EH$ et sur $M_0$ : si $(\delta_0,\delta_\infty)$ est un
couple de poids comme ci-dessus, on d\'efinit $C^{k,\alpha}_{\delta_0,\delta_\infty;t}(M,E)$
par la norme, pour une section $s$ de $E$,
\begin{equation}
  \label{eq:51}
  \| s \|_{C^{k,\alpha}_{\delta_0,\delta_\infty;t}} = t^{\frac{\delta_0+\ell}2}\|\chi_ts\|_{C^{k,\alpha}_{\delta_0}(EH)}
    + \|(1-\chi_t)s\|_{C^{k,\alpha}_{\delta_0,\delta_\infty}(M_0)}.
\end{equation}
Ici $\ell:=\ell_+-\ell_-$ est le poids conforme du fibr\'e
$E\subset TM^{\otimes\ell_+}\otimes T^*M^{\otimes\ell_-}$. C'est le poids qui permet d'identifier les
normes des sections de $E$ pour deux m\'etriques $g$ et $\frac gt$ :
$$ t^{\frac \ell2}|s|_{\frac gt} = |s|_g . $$
Compte tenu que, dans la r\'egion $\EH^t_>$, on a
$$ \rho \sim \frac {\tilde r}{\sqrt t} , $$
le facteur $t^{\frac{\delta_0}2}$ dans (\ref{eq:51}) fait exactement co\"\i ncider
les normes \`a poids de $\EH$ et $M_0$ dans la r\'egion de recollement. En
particulier, les normes des espaces $C^{k,\alpha}_{\delta_0,\delta_\infty;t}(M)$ et
$C^{k,\alpha}_{\delta_0,\delta_\infty}(M_0)$ sont uniform\'ement \'equivalentes pour les sections \`a
support dans $M_t-EH^t$ : pour tout $t>0$, si $s$ est \`a support dans
$M_t-EH^t$, alors
\begin{equation}
 C^{-1} \| s \|_{C^{k,\alpha}_{\delta_0,\delta_\infty}(M_0)}
 \leq \| s \|_{C^{k,\alpha}_{\delta_0,\delta_\infty;t}(M)}
 \leq C \| s \|_{C^{k,\alpha}_{\delta_0,\delta_\infty}(M_0)} ,\label{eq:52}
\end{equation}
pour une constante $C$ ind\'ependante de $t$.

\emph{Dans la suite de l'article, on consid\`ere les poids $\delta_0>0$ et $\delta_\infty>0$
fix\'es et suffisamment petits.}

\section{La jauge de Bianchi}
\label{sec:jauge-Bianchi}

Le lemme suivant montre que la projection $\pi_g$ d\'efinie par
(\ref{eq:45}) existe de mani\`ere tr\`es g\'en\'erale pour les perturbations
de $g_t$ :
\begin{lemm}\label{lem:Bdelta-inv}
  Soit $k\geq 1$ un entier, soient $\delta\in ]-1,3[$ et $\delta'\in ]0,4[$.
  Il existe des constantes $\eta_1,\eta_2>0$, telles que si $t<\eta_1$ et
  \begin{equation}
  \|g-g_t\|_{C^{k,\alpha}_{\delta_0,\delta_\infty;t}}< \eta_2 t^{\frac{\delta_0}2},\label{eq:53}
  \end{equation}
  alors $B_g\delta_g^*:C^{k+1,\alpha}_{\delta,\delta';t}\to C^{k-1,\alpha}_{\delta+2,\delta';t}$ soit inversible, et
  la norme de l'inverse soit born\'ee ind\'ependamment de $t$ et $g$.
\end{lemm}
\begin{proof}
  Tout d'abord, on v\'erifie que, sous l'hypoth\`ese du lemme, toutes les
  m\'etriques $g$ ont des cartes dans lesquelles leurs coefficients se
  comparent de mani\`ere uniforme. Faisons le sur la partie $\EH^t$ de $M$
  (c'est similaire sur la partie $M_0^t$ de $M_0$). Remarquons que sur
  $\EH^t$ on peut tester la condition (\ref{eq:53}) indiff\'eremment par
  rapport \`a $g_t$ ou $t\eh$, puisque $|\nabla^k(\frac{g_t}t-\eh)| =
  O_\eh(t\rho^{2-k})$, donc la condition (\ref{eq:53}) est largement
  satisfaite pour la diff\'erence $t\eh-g_t$.

  Sur un compact fix\'e de $\EH$, la condition (\ref{eq:53}), compte tenu de
  la d\'efinition (\ref{eq:51}), nous dit que
$$ \| \frac gt - \eh \|_{C^{k,\alpha}(\EH)} < \eta_2 . $$
  Si $\eta_2>0$ est assez petit, alors $\frac gt$ se compare uniform\'ement \`a la
  m\'etrique $\eh$.

  Sur la partie non compacte : soit un anneau de la forme
$$ A_R = \{ R \leq r \leq 2R \} . $$
  Consid\'erons l'homoth\'etie $h$ de rapport $R$ qui envoie l'anneau $A_1$
  dans l'anneau $A_R$, alors on v\'erifie imm\'ediatement que 
  $$ \|g-t\eh\|_{C^{k,\alpha}_{\delta_0,\delta_\infty;t}(A_R)}
  \sim t^{\frac{\delta_0}2}R^{\delta_0} \|\frac1{R^2}h^*(\frac gt - \eh)\|_{C^{k,\alpha}(A_1)}.
  $$
  La m\'etrique $\frac 1{R^2}\frac gt$ devient de plus en plus proche de la
  m\'etrique plate de l'anneau $A_R$ quand $R$ devient grand. Sous
  l'hypoth\`ese (\ref{eq:53}), ces m\'etriques sont se comparent uniform\'ement.

  Les estimations elliptiques dans les espaces de H\"older \`a poids se
  montrent justement en se ramenant \`a un nombre fini d'ouverts relativement
  compacts sur lesquels la m\'etrique est uniform\'ement born\'ee par rapport \`a
  une m\'etrique de r\'ef\'erence :
\begin{itemize}
\item un compact de $\EH$ ;
\item un anneau $A_1$ (apr\`es changement d'\'echelle) ;
\item un compact de $M_0$ ;
\item une boule hyperbolique pour les boules centr\'ees sur des points allant
  vers $\partial M_0$.
\end{itemize}
Soit un op\'erateur elliptique g\'eom\'etrique $P$, d'ordre $d\leq k+1$. Sur chacun
de ces ouverts, on dispose d'estimations elliptiques avec constante
uniforme pour l'op\'erateur $P$. Les espaces fonctionnels d\'efinis section
\ref{sec:espaces-fonctionnels} sont justement obtenus \`a partir de ces
ouverts par changement d'\'echelle, et on en d\'eduit imm\'ediatement, si $d\leq \ell\leq
k+1$, l'estimation elliptique
\begin{equation}
  \label{eq:54}
  \|s\|_{C^{\ell,\alpha}_{\delta_0,\delta_\infty;t}} \leq c \big( \|s\|_{C^0_{\delta_0,\delta_\infty;t}} +
  \|Ps\|_{C^{\ell-d,\alpha}_{\delta_0+d,\delta_\infty;t}} \big) ,
\end{equation}
o\`u la constante $c$ peut \^etre choisie ind\'ependante de la m\'etrique $g$
satisfaisant (\ref{eq:53}).

Passons \`a l'assertion sur $B_g\delta_g^*$. Il s'agit de montrer une estimation
  \begin{equation}
    \|X\|_{C^{k+1,\alpha}_{\delta,\delta';t}} \leq c \|B_g\delta_g^*X\|_{C^{k-1,\alpha}_{\delta+2,\delta';t}} ,\label{eq:55}
  \end{equation}
  avec constante ind\'ependante de $g$. Compte tenu de la r\'egularit\'e
  elliptique (\ref{eq:54}), avec constantes uniformes, il suffit de montrer
  une estimation
  \begin{equation}
    \label{eq:56}
    \|X\|_{C^0_{\delta,\delta';t}} \leq c \|B_g\delta_g^*X\|_{C^0_{\delta+2,\delta';t}} .
  \end{equation}
  Le raisonnement se fait classiquement par l'absurde, et on sera donc bref
  : s'il existe
  \begin{itemize}\item des r\'eels $t_i\to0$,
  \item des m\'etriques $g_i$ telles que
    $\|g_i-g_{t_i}\|_{C^{k,\alpha}_{\delta_0,\delta_\infty;t}}< \eta_2$ ,
  \item des champs de vecteurs $X_i$ tels que
    $\|X_i\|_{C^0_{\delta,\delta';t}}=1$ et $\|B_{g_i}\delta_{g_i}^*X\|_{C^0_{\delta+2,\delta';t}}\to0$,
  \end{itemize}
  alors, en regardant le point $x_i\in M$ o\`u la norme $\|X_i\|_{C^0_{\delta,\delta';t}}$
  est atteinte, on extrait de $X_i$ (\'eventuellement apr\`es changement
  d'\'echelle) une limite non nulle $X_\infty$ sur une vari\'et\'e limite $M_\infty$ vers
  laquelle $x_i$ converge.  La limite satisfait $B\delta^*X_\infty=0$, et on montre
  pour les trois limites possibles de $M_\infty$ que c'est impossible :
  \begin{enumerate}
  \item $M_\infty=M_0$ : alors $X_\infty\in C^0_{\delta,\delta'}(M_0)$, donc $X_\infty\in L^2(M_0)$
    ($\delta=2$ n'est pas un poids critique en $p_0$ \`a cause de l'action de
    $\setZ_2$), mais $B_{g_0}\delta_{g_0}^*=\frac12(\nabla^*\nabla-\Ric)$ est injectif
    sur $L^2(M_0)$ ;
  \item $M_\infty=\EH$ : alors $X_\infty\in C^0_{\delta}(\EH)$, \`a nouveau
    $B_\eh\delta_\eh^*=\frac12 \nabla^*\nabla$ n'a pas de noyau sur $C^0_{\delta}(M_0)$ ;
    \`a nouveau c'est l'action de $\setZ_2$ qui interdit le poids critique
    $0$ \`a l'infini ;
  \item $M_\infty=\setC^2/\setZ_2$ (la partie interm\'ediaire entre $\EH$ et $M_0$) : m\^eme
    argument.
  \end{enumerate}
\end{proof}

On en d\'eduit en particulier que toute m\'etrique assez proche de $g_t$
peut se mettre en jauge de Bianchi par rapport \`a $g_t$ :
\begin{lemm}\label{lemm:mise-en-jauge}
  Soit $k\geq 1$ un entier. Il existe des constantes $\eta,\epsilon,\epsilon'>0$, tels que pour
  tout $t<\eta$ et toute m\'etrique $g$ sur $M$ telle que
  $\|g-g_t\|_{C^{k,\alpha}_{\delta_0,\delta_\infty;t}}<\epsilon t^{\frac{\delta_0}2}$, il existe un
  diff\'eomorphisme $\varphi=\exp_{g_t}X$ avec $\|X\|_{C^{k+1,\alpha}_{\delta_0-1,\delta_\infty;t}} < \epsilon'
  t^{\frac{\delta_0}2}$, tel que $B_{g_t}\varphi^*g=0$.
\end{lemm}
  C'est une application du th\'eor\`eme des fonctions inverses, sous la forme
  explicite donn\'ee par le lemme suivant :
\begin{lemm}\label{lem:pt-fixe}
  Soit $\Phi:E\to F$ une application lisse entre espaces de Banach, et soit
  $Q=\Phi-\Phi(0)-d_0\Phi$. Supposons qu'existent des constantes strictement positives
  $q$, $r_0$ et $c$ telles que
  \begin{enumerate}
  \item $\|Q(x)-Q(y)\| \leq q\|x-y\|(\|x\|+\|y\|)$ pour tous $x,y\in B(0,r_0)$;
  \item $d_0\Phi$ est un isomorphisme, et $\|(d_0\Phi)^{-1}\|\leq c$.
  \end{enumerate}
  Si $r\leq \min(r_0,\frac1{2qc})$ et $\|\Phi(0)\| \leq \frac r{2c}$, alors l'\'equation
  $\Phi(x)=0$ admet une solution unique dans $B(0,r)$.\qed
\end{lemm}
\begin{proof}[D\'emonstration du lemme \ref{lemm:mise-en-jauge}]
  On applique le lemme \ref{lem:pt-fixe} \`a l'op\'erateur
$$ X \longmapsto \Phi(X,g) = B_{(\exp_{g_t}X)^*g_t}g , $$
  entre les espaces
  \begin{equation}
t^{-\frac{\delta_0}2}C^{k+1,\alpha}_{\delta_0-1,\delta_\infty;t}
 \longrightarrow t^{-\frac{\delta_0}2}C^{k-1,\alpha}_{\delta_0+1,\delta_\infty;t}.\label{eq:57}
  \end{equation}

  Le facteur $t^{-\frac{\delta_0}2}$ s'explique de la mani\`ere suivante : sur un
  grand compact $K\subset\EH$, on a
  \begin{align*}
    t^{-\frac{\delta_0}2}\|X\|_{C^{k+1,\alpha}_{\delta_0-1,\delta_\infty;t}} &\sim \|X\|_{C^{k+1,\alpha}(\eh)},\\
    t^{-\frac{\delta_0}2}\|X\|_{C^{k-1,\alpha}_{\delta_0+1,\delta_\infty;t}} &\sim t \|X\|_{C^{k-1,\alpha}(\eh)}. 
  \end{align*}
  Sachant que
  $$ B_{(\exp_{g_t}X)^*g_t}g=tB_{(\exp_{\frac{g_t}t}X)^*\frac{g_t}t}\frac gt,$$
  on voit que sur $K$ l'op\'erateur $\Phi$ entre les espaces (\ref{eq:57}) est
  identique \`a l'op\'erateur $X\mapsto B_{(\exp_{\frac{g_t}t}X)^*\frac{g_t}t}\frac
  gt$ entre les espaces standards $C^{k+1,\alpha}\to C^{k-1,\alpha}$ pour la m\'etrique de
  Eguchi-Hanson.

  Sur chaque ouvert o\`u la m\'etrique se ram\`ene, par changement d'\'echelle, \`a
  une m\'etrique uniform\'ement proche d'un mod\`ele (voir la d\'emonstration du
  lemme \ref{lem:Bdelta-inv}), on peut faire la m\^eme op\'eration et v\'erifier
  que l'op\'erateur $\Phi$ se ram\`ene, apr\`es changement d'\'echelle, au m\^eme
  op\'erateur pour la m\'etrique contr\^ol\'ee, entre des espaces
  \begin{equation}
  \lambda C^{k+1,\alpha}\to\lambda C^{k-1,\alpha},\label{eq:58}
  \end{equation}
  avec $\lambda\geq 1$ (le cas ci-dessus o\`u $\lambda=1$ \'etant le plus d\'efavorable).

  La g\'eom\'etrie ainsi bien contr\^ol\'ee, on peut passer \`a la d\'emonstration du
  lemme : on a $\Phi(X,g)=0$ si et seulement si $B_{g_t}(\exp_{g_t}X)_*g=0$,
  et
$$ \frac{\partial\Phi}{\partial X}\Big|_{(0,g_t)}(X) = B_{g_t}\delta_{g_t}^*X , $$
qui, par le lemme \ref{lem:Bdelta-inv}, est inversible, avec constante
contr\^ol\'ee ind\'ependamment de $t$. Si $\|g-g_t\|_{C^{k,\alpha}_{\delta_0,\delta_\infty;t}}< \epsilon
t^{\frac{\delta_0}2}$, et $\epsilon$ est assez petit, alors la norme d'op\'erateur
$\|\frac{\partial\Phi}{\partial X}(0,g_t)-\frac{\partial\Phi}{\partial X}(0,g)\|$ est petite (c'est une
question locale qui se voit dans les cartes ci-dessus), donc $\frac{\partial\Phi}{\partial
  X}(0,g)$ demeure inversible, et la norme de l'inverse est contr\^ol\'ee. En
outre,
\begin{equation}
 \|\Phi(0,g)\|_{C^{k-1,\alpha}_{\delta_0+1,\delta_\infty;t}} \leq c \epsilon t^{\frac{\delta_0}2}.\label{eq:59}
\end{equation}
On applique le lemme \ref{lem:pt-fixe} \`a l'op\'erateur $X\to\Phi(X,g)$ entre les
espaces (\ref{eq:57}), en \'ecrivant
$$ \Phi(X,g) = \Phi(0,g) + \frac{\partial\Phi}{\partial X}\big|_{(0,g)}(X) + Q(X) .$$
Il reste \`a montrer l'estimation 
$$ \| Q(X)-Q(Y) \|_{t^{-\frac{\delta_0}2}C^{k-1,\alpha}_{\delta_0+1,\delta_\infty;t}}
\leq q \|X-Y\|_{t^{-\frac{\delta_0}2}C^{k+1,\alpha}_{\delta_0-1,\delta_\infty;t}}
    \|X+Y\|_{t^{-\frac{\delta_0}2}C^{k+1,\alpha}_{\delta_0-1,\delta_\infty;t}} .
$$
C'est une estimation locale, qu'il suffit de v\'erifier dans chacune des
cartes o\`u $\Phi$ se ram\`ene \`a un op\'erateur entre les espaces de H\"older
standards (\ref{eq:58}). Comme $Q$ repr\'esente des termes au moins
quadratiques, une telle estimation est imm\'ediate entre les espaces de
H\"older standards. Quand on multiplie les normes par $\lambda$ comme dans
(\ref{eq:58}), l'in\'egalit\'e continue d'aller dans le bon sens si $\lambda\geq 1$,
ce qui est justement le cas. On notera qu'obtenir exactement $\lambda\geq 1$ avec un
minimum \'egal \`a 1 est la raison d'\^etre du facteur $t^{-\frac{\delta_0}2}$ dans
(\ref{eq:57}).

La constante $q$ du lemme \ref{lem:pt-fixe} est donc contr\^ol\'ee.
Si on prend $\epsilon$ assez petit dans (\ref{eq:59}), le lemme
\ref{lem:pt-fixe} fournit donc une unique solution.
\end{proof}

\section{R\'esolution de l'\'equation modulo les obstructions}
\label{sec:resol-mod-O}

Le lemme \ref{lem:Bdelta-inv} donne un sens au probl\`eme $\Phi_{g_t}(g)\in
\cO_g$ quand $g$ est une perturbation de $g_t$ dans
$C^{2,\alpha}_{\delta_0,\delta_\infty}$. Avant de r\'esoudre le probl\`eme, d\'efinissons pour
$\upsilon=(\upsilon_1,\upsilon_2,\upsilon_3)\in \setR_3$ assez petit une m\'etrique
\begin{equation}
  \label{eq:60}
  g_{t,\upsilon} = g_t + t \chi_t \sum_1^3 \upsilon_j o_j .
\end{equation}
Nous utilisons ainsi explicitement l'ambigu\"\i t\'e de la solution dans la
proposition \ref{pro:sol-EH-O}. Remarquons que
\begin{equation}
\|t o_j\|_{C^{2,\alpha}_{\delta_0,\delta_\infty;t}}=O(t^{1+\frac{\delta_0}4}),\label{eq:61}
\end{equation}
donc $g_{t,\upsilon}$ est une petite perturbation de $g_t$.

\begin{prop}\label{prop:resol-O}
  Soit $(M_0,g_0)$ une m\'etrique orbifold Einstein, avec singularit\'e de type
  $\setR^4/\setZ_2$ au point $p_0$. Supposons $M_0$ compacte, ou alors non compacte
  AH ; dans tous les cas non d\'eg\'en\'er\'ee. Alors il existe $\eta,\epsilon>0$ tels que,
  pour $t<\eta$ et $\upsilon=(\upsilon_1,\upsilon_2,\upsilon_3)$ avec $|\upsilon|<\epsilon$, il existe une unique
  solution $\hat g_t$ de l'\'equation
$$ \Ric_{\hat g_t}-\Lambda \hat g_t \in \cO_{\hat g_t} , $$
  satisfaisant les conditions suivantes :
  \begin{enumerate}
    \item $\|\hat g_t-g_t\|_{C^{2,\alpha}_{\delta_0,\delta_\infty}} \leq \epsilon t^{\frac{\delta_0}2}$ ;
    \item $\hat g_t$ est en jauge de Bianchi par rapport \`a $g_t$,
      c'est-\`a-dire $B_{g_t}\hat g_t=0$ ;
    \item pour $j=1,2,3$, on a sur la section nulle dans $\EH$ :
    \begin{equation}\label{eq:62}
      \int_{S^2} \langle \tfrac1t(\hat g_t-g_t),o_j\rangle_\eh \vol_\eh = \upsilon_j.
    \end{equation}
  \end{enumerate}
  En outre, \'ecrivant $\Ric_{\hat g_t}-\Lambda \hat g_t=\sum_1^3\lambda_j(t)\pi_{\hat
    g_t}(\chi_to_j)$, on a les estimations
  \begin{align}
    \label{eq:63}
    \| \hat g_t-g_{t,\upsilon} \|_{C^{2,\alpha}_{\delta_0,\delta_\infty;t}} &\leq c t^{1+\frac{\delta_0}4} , \\
    \label{eq:64}
    \lambda_j(t) &= t \lambda_j + O(t^{\frac32-\delta})\text{ pour tout }\delta>0.
  \end{align}
  Les estimations sont uniformes par rapport \`a $\upsilon$ et \`a une variation de
  l'infini conforme et donc de la m\'etrique orbifold d'Einstein $g_0$.
\end{prop}
Quand nous faisons varier la m\'etrique orbifold Einstein $g_0$ \`a partir
d'une variation $\gamma$ de son infini conforme $\gamma_0$, nous notons explicitement
la d\'ependance des objets intervenant dans la proposition par rapport \`a la
jauge $\varphi\in \Sp_1/U_1$ et \`a $\gamma$ : $\lambda_j(\varphi,\gamma)$, $\lambda_j(t,\varphi,\gamma)$, $g_{t,\varphi,\upsilon,\gamma}$, $\hat
g_{t,\varphi,\upsilon,\gamma}$\ldots Le param\`etre $\upsilon$ sera rapidement oubli\'e car il ne sert en
r\'ealit\'e que dans le corollaire \ref{cor:obs-E} ci-dessous.

\begin{rema}\label{rem:sol-avec-gt0}
  On peut remarquer que, par rapport \`a la m\'etrique recoll\'ee plus na\"\i ve
  $g_t^0$ donn\'ee par (\ref{eq:41}), on a
  \begin{equation}
  \|g_t-g_t^0\|_{C^{2,\alpha}_{\delta_0,\delta_\infty;t}}=O(t^{\frac12+\frac{\delta_0}4}).\label{eq:65}
  \end{equation}
  Il en r\'esulte imm\'ediatement que les (1) (2) et (3) de la proposition
  restent vrais en rempla\c cant $g_t$ par $g_t^0$. On obtient un gain en
  rempla\c cant $g_t^0$ par la meilleure solution approch\'ee $g_t$ dans
  les estimations (\ref{eq:63}) et (\ref{eq:64}) sur la solution.
\end{rema}

On peut d\'eduire de la proposition un \'enonc\'e montrant que la condition de
courbure sur $g_0$ au point $p_0$ est une condition n\'ecessaire pour
r\'esoudre l'\'equation d'Einstein par notre m\'ethode de recollement :
\begin{coro}\label{cor:obs-E}
  Sous les hypoth\`eses de la proposition \ref{prop:resol-O},
  si la condition de courbure $\det \bR^+_{g_0}(p_0)=0$ n'est pas
  satisfaite, alors il existe $\epsilon,\eta>0$ tels que si $\gamma$ est assez proche de
  $\gamma_0$ et $t<\eta$, alors il n'existe pas de jauge $\varphi\in \Sp_1/U_1$ et de m\'etrique
  d'Einstein $g$ telle que
  $$ \|g-g^0_{t,\varphi,\gamma}\|_{C^{1,\alpha}_{\delta_0,\delta_\infty;t}} \leq \epsilon t^{\frac{\delta_0}2} . $$
\end{coro}
L'utilisation de la m\'etrique recoll\'ee na\"\i ve $g_{t,\varphi,\gamma}^0$ dans l'\'enonc\'e
permet d'avoir les hypoth\`eses les plus simples possibles.

\begin{proof}[D\'emonstration du corollaire \ref{cor:obs-E}]
  Pour tous les $\gamma$ suffisamment proches de $\gamma_0$ et tous les $\varphi\in \Sp_1$,
  les $\lambda_i(\varphi,\gamma)$ ne sont pas tous nuls, par exemple
  $$ |\lambda_1(\varphi,\gamma)| > \epsilon \text{ pour un }\epsilon>0. $$
  Si $g$ est une m\'etrique d'Einstein, proche de $g^0_{t,\varphi,\gamma}$ au sens du du
  corollaire, alors elle est \'egalement proche de $g_{t,\varphi,\gamma}$ par
  (\ref{eq:65}), donc, par le lemme \ref{lemm:mise-en-jauge}, peut \^etre
  mise en jauge de Bianchi par rapport \`a $g_{t,\varphi,\gamma}$.

  La r\'egularit\'e de l'\'equation d'Einstein en jauge de Bianchi permet de
  d\'eduire qu'en r\'ealit\'e, on dispose d'une borne
  $$  \|g-g_{t,\varphi,\gamma}\|_{C^{2,\alpha}_{\delta_0,\delta_\infty;t}} \leq \epsilon t^{\frac{\delta_0}2} . $$
  Toujours par l'hypoth\`ese, on a alors $ |\int_{S^2} \langle
  \tfrac1t(g-g_t),o_j\rangle_\eh \vol_\eh|<\epsilon$, donc on peut appliquer l'unicit\'e
  dans la proposition \ref{prop:resol-O} pour conclure que $g$ co\"\i ncide
  avec l'une des m\'etriques $\hat g_{t,\varphi,\upsilon,\gamma}$.  Mais, pour $t<\eta$ avec $\eta$
  assez petit, uniforme pour tous les infinis conformes $\gamma$ proches de
  $\gamma_0$, la m\'etrique $\hat g_{t,\varphi,\upsilon,\gamma}$ ne peut pas \^etre Einstein \`a cause
  de (\ref{eq:64}).
\end{proof}

\begin{proof}[D\'emonstration de la proposition \ref{prop:resol-O}]
  On consid\`ere l'op\'erateur
  \begin{equation}
 \Phi(h,x) = \Ric_g - \Lambda g+\delta_g^*B_{g_t}g - \sum_1^3x_j\pi_g(\chi_to_j) ,
 \quad \text{ o\`u }g=g_{t,\upsilon}+h ,\label{eq:66}
  \end{equation}
  vu comme agissant entre les espaces
  \begin{equation}
 t^{-\frac{\delta_0}2} \ ^0 C^{2,\alpha}_{\delta_0,\delta_\infty;t} \oplus \setR^3
\longrightarrow t^{-\frac{\delta_0}2} C^\alpha_{\delta_0+2,\delta_\infty;t},\label{eq:67}
  \end{equation}
  l'espace $\ ^0C^{2,\alpha}$ d\'efini par la condition (\ref{eq:62}) sur
  $g_{t,\upsilon}=g-h$.

  La proposition se ram\`ene \`a r\'esoudre le probl\`eme $\Phi(h,x)=0$, et \`a en
  estimer la solution. En effet, si on a une solution $(h,x)$, par les
  lemmes \ref{lem:prob-mod-O} et \ref{lem:Bdelta-inv}, la m\'etrique
  $\hat g_t=g_t+h$ satisfait alors l'\'equation $(\Ric-\Lambda)\hat g_t\in
  \cO_{\hat g_t}$ tout en \'etant en jauge de Bianchi par rapport \`a
  $g_t$.

  \`A nouveau, il s'agit de voir qu'on peut appliquer le lemme
  \ref{lem:pt-fixe}, avec des constantes $c$ et $q$ uniform\'ement
  born\'ees. Comme, par le lemme \ref{lem:est-ini-Ric} et l'estimation
  (\ref{eq:61}), on a au d\'epart
  $$ \| \Phi(0) \|_{C^\alpha_{\delta_0+2,\delta_\infty;t}} = O(t^{1+\frac{\delta_0}4}), $$
  alors pour $t$ assez petit on peut appliquer le lemme
  \ref{lem:pt-fixe}, et l'estimation (\ref{eq:63}) en r\'esultera
  automatiquement. Dans la suite, le param\`etre $\upsilon$ ne joue plus aucun r\^ole
  et on l'omettra.

La diff\'erentielle de l'op\'erateur $\Phi$ s'\'ecrit
\begin{equation}\label{eq:68}
 d_0\Phi(h,x)=\Psi(h,x)+\frac12(\Ric_{g_t}\circ h+h\circ \Ric_{g_t}-2\Lambda h),
\end{equation}
o\`u 
\begin{equation}
  \begin{aligned}
    \Psi(h,x)&=P_th+\sum_1^3 x_j\pi_{g_t}(\chi_to_j), \\
    P_t h&=\frac12\nabla_{g_t}^*\nabla_{g_t} h-\rR_{g_t} h.\label{eq:69}
  \end{aligned}
\end{equation}
Compte tenu de $\|\Ric_{g_t}-\Lambda g_t\|_{C^{2,\alpha}_{\delta_0+2,\delta_\infty}} =
  O(t^{1+\frac{\delta_0}4})$, l'op\'erateur $d_{(g_t,0)}\Phi$ est une petite
  perturbation de $\Psi$ et il nous suffira de borner inf\'erieurement la norme
  de $\Psi^{-1}$ entre les espaces apparaissant dans (\ref{eq:67}).

  On notera que $\|\chi_to_j\|_{C^\alpha_{\delta_0,\delta_\infty;t}} \sim t^{\frac{\delta_0}2}$, donc
  la norme standard sur $\setR^3$ est exactement la norme ad\'equate dans
  (\ref{eq:67}). On en d\'eduit, comme dans la d\'emonstration du lemme
  \ref{lemm:mise-en-jauge}, que tout l'op\'erateur $\Phi$ est lisse avec
  des coefficients contr\^ol\'ees dans les cartes locales par la norme de
  $t^{-\frac{\delta_0}2}C^{2,\alpha}_{\delta_0,\delta_\infty;t}\oplus\setR^3$. Il en r\'esulte que les
  termes non lin\'eaires de l'op\'erateur $\Phi$ satisfont l'estimation
  uniforme requise dans la condition (1) du lemme \ref{lem:pt-fixe}.

  Pour appliquer ce lemme, il reste donc \`a montrer que l'op\'erateur
  $d_0\Phi$ admet un inverse born\'e uniform\'ement, donc \`a montrer une
  estimation
  \begin{equation}
    \label{eq:70}
    t^{-\frac{\delta_0}2} \|h\|_{C^{2,\alpha}_{\delta_0,\delta_\infty;t}} + |x|
    \leq c t^{-\frac{\delta_0}2} \| \Psi(h,x) \|_{C^\alpha_{\delta_0+2,\delta_\infty;t}}
  \end{equation}
  pour une constante $c$ ind\'ependante de $t$. Le raisonnement se fait
  par contradiction, comme dans le lemme
  \ref{lem:Bdelta-inv}. Supposons donn\'es une suite $t_i\to0$, et
  $(h_i,x_i)$ tels que
  $$ t_i^{-\frac{\delta_0}2} \|h_i\|_{C^0_{\delta_0,\delta_\infty;t_i}} + |x_i| = 1, $$
  mais
  $$ t_i^{-\frac{\delta_0}2} \| \Psi(h_i,x) \|_{C^0_{\delta_0+2,\delta_\infty;t_i}} \to 0 . $$
  
  Supposons dans un premier temps que $x_i\to0$. Alors le raisonnement
  est tr\`es similaire \`a celui fait pour le lemme
  \ref{lem:Bdelta-inv} : la norme est atteinte \`a un point $p_i$,
  et on extrait, \'eventuellement apr\`es changement d'\'echelle, une limite
  $h_\infty$ de $h_i$ sur une vari\'et\'e limite $M_\infty$ vers laquelle $p_i$
  converge. Les trois cas sont les m\^emes :
  \begin{enumerate}
  \item $M_\infty=M_0$ : alors $\|h_\infty\|_{C^0_{\delta_0,\delta_\infty}}=1$ et $P_{g_0}h_\infty=0$ ; cela
    impliquerait que $h_\infty$ soit $L^2$, mais par hypoth\`ese, $P$ n'a pas
    de noyau sur $M_0$, contradiction ;
  \item $M_\infty=\EH$ : alors $\|h_\infty\|_{C^0_{\delta_0}}=1$ et $P_\eh h_\infty=0$, et
    en outre on a (\ref{eq:62}) ; or le noyau de $P_\eh$ dans
    $C^0_{\delta_0}$ consiste pr\'ecis\'ement des $o_j$, mais ceux-ci sont
    interdits par la condition (\ref{eq:62}), d'o\`u la contradiction ;
  \item $M_\infty=\setC^2/\setZ_2$ : alors $\|h_\infty\|_{C^0_{\delta_0}}=1$ et $Ph_\infty=\frac12
    \nabla^*\nabla h_\infty=0$, mais la condition $h_\infty\in C^0_{\delta_0}$ implique que $h_\infty$
    soit lisse \`a l'origine (au sens orbifold), et tende vers $0$ \`a
    l'infini : il n'existe pas de telle solution, d'o\`u la contradiction.
  \end{enumerate}

  Finalement nous avons r\'eduit la d\'emonstration \`a \'etablir
  $x_i\to0$. Cela sera \'etabli comme cons\'equence des faits suivants sur
  les produits scalaires $L^2$ (il faut noter que le produit scalaire
  de 2-tenseurs sym\'etriques est invariant conforme) :
  \begin{align}
    \label{eq:71}
    \big|\big(u,\chi_to_j\big)_{g_t}\big|
      &\leq c t^{-\frac{\delta_0}2} \|u\|_{C^\alpha_{\delta_0,\delta_\infty;t}},\\
    \label{eq:72}
    \big|\big(P_th,\chi_to_j\big)_{g_t}\big|
      &\leq c t^{\frac12-\frac{\delta_0}4} \|h\|_{C^{2,\alpha}_{\delta_0,\delta_\infty;t}},\\
    \label{eq:73}
    \big|\big(\pi_{g_t}(\chi_to_j),\chi_to_k\big)_{g_t}\big|
      &= (o_j,o_k)_\eh + O(t^{\frac12-\delta}) \text{ pour tout }\delta>0.
  \end{align}
  En effet, de (\ref{eq:72}) et (\ref{eq:73}) on d\'eduit, puisque
  $t^{-\frac{\delta_0}2} \|h_i\|_{C^{2,\alpha}_{\delta_0,\delta_\infty;t}}\leq 1$ et
  $(o_j,o_k)_\eh= \|o_1\|_\eh^2\delta_{jk}$,
  $$ \|o_1\|_\eh^2 x_{i,j} = \big(\Psi(h_i,x_i),\chi_t o_j\big)_{g_t} + O(t^{\frac12-\delta}) ,
  $$
  donc, par (\ref{eq:71}), pour $j=1$, $2$, $3$,
  $$ |x_{i,j}| \leq c t^{-\frac{\delta_0}2} \|\Psi(h_i,x_i)\|_{C^\alpha_{\delta_0,\delta_\infty;t}} + O(t^{\frac12}).$$
  Cela prouve bien $x_i\to0$ comme voulu.

  Il nous reste donc \`a d\'emontrer les estimations (\ref{eq:71}),
  (\ref{eq:72}) et (\ref{eq:73}). L'estimation (\ref{eq:71}) s'obtient
  imm\'ediatement par int\'egration \`a partir de la d\'efinition de la norme
  de l'espace $C^\alpha_{\delta_0,\delta_\infty;t}$. Pour l'estimation (\ref{eq:72}), on
  fait l'int\'egration par parties
  \begin{equation}
  \big(P_th,\chi_to_j\big)_{g_t}
  = \big(h,P_t(\chi_t o_j)\big)_{g_t}
  = \big(\tfrac ht,P_{\frac{g_t}t}(\chi_t o_j)\big)_{\frac{g_t}t} .\label{eq:74}
  \end{equation}
  La seconde \'egalit\'e est juste le r\'esultat d'une homoth\'etie sur la m\'etrique.
  Comme les m\'etriques $\frac{g_t}t$ et $\eh$ sont uniform\'ement \'equivalentes
  sur $\EH^t$, on peut estimer le second membre par rapport \`a $\eh$. Or
  \begin{align}
    |P_{\frac{g_t}t}(\chi_t o_j)|
    &\leq |(P_{\frac{g_t}t}-P_\eh)(\chi_t o_j)| + |P_\eh(\chi_t o_j)| \notag \\
    &\leq c ( t \rho^{-4} + 1_{\EH^t_>} \rho^{-6} ).\label{eq:75}
  \end{align}
  Int\'egrant contre $h$, sachant que $|\frac ht|_\eh\leq
  t^{-\frac{\delta_0}2}\rho^{-\delta_0}\|h\|_{C^{2,\alpha}_{\delta_0,\delta_\infty;t}}$, on d\'eduit
  $$ \big|\big(P_th,\chi_to_j\big)_{g_t}\big|
    \leq c t^{\frac12-\frac{\delta_0}4}\|h\|_{C^{2,\alpha}_{\delta_0,\delta_\infty;t}},$$
  ce qui prouve bien (\ref{eq:72}).

  Enfin, pour la derni\`ere estimation (\ref{eq:73}), remarquons tout
  d'abord que, puisque $B_{g_t}=\frac1{t^2} B_{\frac{g_t}t}$, on a sur $EH^t$
  \begin{equation}
   |B_{g_t}(\chi_t o_j)|_{\frac{g_t}t}
    \leq c t^{-2} \big( \frac{1_{EH^t_>}}{\rho^5} + \frac t{\rho^3} \big),\label{eq:76}
  \end{equation}
  et par cons\'equent,
  \begin{equation}
   \| B_{g_t}(\chi_t o_j) \|_{C^{k,\alpha}_{\delta+2}}
    = t^{\frac{\delta+3}2} \sup \rho^{2+\delta}|B_{g_t}(\chi_t o_j)|_{\frac{g_t}t}
    \leq c t^{\frac \delta4}.\label{eq:77}
  \end{equation}
  Si $-1<\delta<3$, il en r\'esulte, par le lemme \ref{lem:Bdelta-inv},
  \begin{equation}
   \| \delta_{g_t}^*(B_{g_t}\delta_{g_t}^*)^{-1}B_{g_t}(\chi_t o_j) \|_{C^{k+1,\alpha}_{\delta+1}}
    \leq c t^{\frac \delta4}.\label{eq:78}
  \end{equation}
  En int\'egrant contre $\chi_to_k$, il en r\'esulte, d\`es que $\delta>0$,
  \begin{equation}
   \big|\big(
       \delta_{g_t}^*(B_{g_t}\delta_{g_t}^*)^{-1}B_{g_t}(\chi_t o_j),\chi_t o_k
     \big)\big| \leq c t^{\frac12-\frac \delta4}.\label{eq:79}
   \end{equation}
  Les autres termes dans (\ref{eq:73}) viennent du fait que la
  m\'etrique par rapport \`a laquelle on mesure le produit scalaire est
  $\frac{g_t}t$ plut\^ot que $\eh$, et le terme de reste
  $\int_{r>t^{-\frac14}}|o_j|^2$, mais tous sont $O(t)$ et on en d\'eduit
  le r\'esultat.

  Passons maintenant \`a la d\'emonstration de l'estimation (\ref{eq:64}).
  On part de l'estimation du lemme \ref{lem:est-ini-Ric},
  $$ \big| (\Ric-\Lambda)(g_t) - \sum_1^3 t \lambda_j(\gamma) \chi_to_j \big|_\eh
    \leq c \rho^2 t^2 , $$
  avec les estimations similaires pour les d\'eriv\'ees. En particulier,
  \begin{equation}
    \label{eq:80}
    \big| \big( (\Ric-\Lambda)(g_t) - \sum_1^3 t \lambda_j(\gamma) \chi_to_j,\chi_to_k \big)_{g_t} \big|
   \leq c t^{\frac32}.
  \end{equation}
  On \'ecrit la solution de l'\'equation d'Einstein avec obstruction comme
  $$ \hat g_t = g_t + j_t , \quad \|j_t\|_{C^{2,\alpha}_{\delta_0,\delta_\infty;t}}\leq ct^{1+\frac{\delta_0}4}. $$
  Par l'estimation (\ref{eq:72}), on obtient
  \begin{equation}
    \label{eq:81}
    \big| (P_tj_t,\chi_to_j)_{g_t} \big| \leq c t^{\frac32} .
  \end{equation}
  Par ailleurs, les estimations (\ref{eq:78}) et (\ref{eq:79}) restent
  valables pour $\hat g_t$, et on en d\'eduit, pour tout $\delta>0$ petit,
  \begin{equation}
    \label{eq:82}
    \big|\big(
       \delta_{\hat g_t}^*(B_{\hat g_t}\delta_{\hat g_t}^*)^{-1}B_{\hat g_t}(\chi_t to_j),\chi_t o_k
     \big)_{g_t}\big| \leq c t^{\frac32-\delta}.
  \end{equation}
  Enfin, observons que
  \begin{equation}
   (\Ric-\Lambda)(\hat g_t)=(\Ric-\Lambda)(g_t)+P_t(j_t)+Q_t(j_t),\label{eq:83}
 \end{equation}
  o\`u $Q_t$ contient les termes au moins quadratiques de $\Ric-\Lambda$. Or on a vu
  $$ \|Q_t(j_t)\|_{C^\alpha_{\delta_0+2,\delta_\infty;t}}
    \leq c t^{-\frac{\delta_0}2} \|j_t\|_{C^{2,\alpha}_{\delta_0,\delta_\infty;t}}^2
    \leq c' t^2,$$
  d'o\`u r\'esulte l'estimation
  \begin{equation}
    \label{eq:84}
    \big|\big( Q_tj_t,\chi_to_j \big)_{g_t}\big| \leq c t^{2-\frac{\delta_0}2} .
  \end{equation}
  Finalement, mettant ensemble (\ref{eq:81}), (\ref{eq:82}) et
  (\ref{eq:84}) dans l'\'equation (\ref{eq:83}), nous obtenons
  $$ \big|\big( (\Ric-\Lambda)(\hat g_t)-\sum_1^3t\lambda_j(\gamma)\pi_{\hat g_t}(\chi_to_j),\chi_t o_j \big)_{g_t}\big|
    \leq  c t^{\frac32-\delta}. $$
  Compte tenu de (\ref{eq:73}), \'egalement valable pour $\hat g_t$, on en
  d\'eduit le lemme.

  Enfin, l'uniformit\'e des estimations par rapport \`a une variation de
  l'infini conforme est claire, puisque $g_0(\gamma)$ (et en particulier
  $\lambda_j(\gamma)$) d\'epend de mani\`ere lisse de $\gamma$, et tous les op\'erateurs
  diff\'erentiels d\'ependent aussi de mani\`ere lisse de $\gamma$.
\end{proof}
\begin{rema}
  L'estimation (\ref{eq:72}) indique qu'un \'eventuel inverse de 
  $P_t:C^{2,\alpha}_{\delta_0,\delta_\infty;t}\to C^\alpha_{\delta_0+2,\delta_\infty;t}$ ne peut pas \^etre born\'e
  sur $\chi_to_j$.
\end{rema}

\section{Solutions singuli\`eres sur l'orbifold}
\label{sec:solutions-orbifold}

Dans cette section, on commence \`a utiliser la non compacit\'e de $M_0$, pour
y construire des solutions particuli\`eres des \'equations
lin\'earis\'ees, avec terme principal singulier donn\'e au point
orbifold. Le but est le lemme \ref{lem:S-inj}, qui garantit l'injectivit\'e
de l'op\'erateur associant au terme principal au point singulier le terme
ind\'etermin\'e \`a l'infini. Celui-ci sera appliqu\'e dans la section
\ref{sec:extension-de-germes} pour \'etendre par dualit\'e des germes de
solutions sur l'orbifold, et dans la section \ref{sec:petit-vale-propr}
pour construire des vecteurs propres approch\'es de $P_t$.

Les deux op\'erateurs que nous avons en vue sont l'op\'erateur
$$B\delta^*=\frac12(\nabla^*\nabla-\Ric)$$ sur $TM_0$, et l'op\'erateur $$P=\frac12(\nabla^*\nabla-\rR)$$
sur $\Sym^2_0T^*M$. Le lien entre les deux op\'erateurs provient de la
formule $P\delta^*=\delta^*B\delta^*$, qui garantit qu'\`a une solution $X$ de
$B\delta^*X=0$ correspond une solution $h=\delta^*X$ de $Ph=0$.

D'une mani\`ere g\'en\'erale, soit $E$ un fibr\'e tensoriel, muni d'un laplacien
g\'eom\'etrique $P=\nabla^*\nabla+\cR$, o\`u $\cR$ est un op\'erateur d'ordre 0 (typiquement,
un terme de courbure). Un voisinage du point orbifold $p_0$ s'identifie \`a
$(]0,r_{p_0}[\times S^3)/\setZ_2$, on trivialise $E$ sur ce voisinage, alors les
termes principaux de $P$ \`a l'origine sont ceux du laplacien de l'espace
plat $\setR^4$, \`a savoir 
\begin{equation}
 -\frac1{r^3}\partial_r(r^3 \partial_r)-\frac1{r^2}\Delta_{S^3} ;\label{eq:85}
\end{equation}
en particulier on obtient comme solution approch\'ee les
\begin{equation}
 r^k f_k, \quad r^{-2-k} f_k, \quad
 \text{ o\`u }\Delta_{S^3}f_k=kf_k.\label{eq:86}
\end{equation}
Les valeurs propres $k$ du laplacien de la sph\`ere $S^3$ sont des
entiers naturels, mais compte tenu de l'invariance sous $\setZ_2$, on a en fait
$$k\in k_0+2\setN, \text{ o\`u }k_0=
\begin{cases}
0 & \text{si le poids conforme de $E$ est pair,}\\
1 & \text{s'il est impair.}
\end{cases}
$$
Notons $F_k$ l'espace propre du laplacien scalaire sur la sph\`ere
$S^3$, constitu\'e de fonctions polynomiales de degr\'e $k$.

Supposons $P$ inversible dans $L^2(M_0)$. \'Etant donn\'e $\sigma_{k_0}\in E_{p_0}\otimes
F_{k_0}$, compte tenu de (\ref{eq:86}), la section
$s'=\frac{\sigma_{k_0}}{r^{2+k_0}}$ est une solution approch\'ee de l'\'equation
$Ps'=0$ pr\`es de $p_0$, et plus pr\'ecis\'ement
\begin{equation}
  \label{eq:87}
  Ps'=O(\frac1{r^{2+k_0}}).
\end{equation}
Par inversibilit\'e de $P$, on peut alors r\'esoudre $Ps''=-Ps'$ avec
$s''=O(\frac1{r^{k_0}})$ et on obtient une solution $s=s'+s''$ de $Ps=0$
qui satisfait $s\sim s'$ pr\`es de $p_0$, et $s$ est $L^2$ \`a l'infini.

Le proc\'ed\'e s'\'etend en partant d'une solution approch\'ee
$s_0=r^{-2-k}\sigma_k$, avec $\sigma_k\in E_{p_0}\otimes F_k$. Alors $Ps_0\sim r^{-2-k}$, et
on peut corriger $s_0$ par un $s_1\sim r^{-k}$ (ou bien $s_1\sim r^{-k}\ln
r$) ; on construit ainsi une somme finie $s'=s_0+s_1+\cdots +s_{[k/2]}$,
solution de $Ps'=O(r^{-2-k_0})$, qu'on peut \`a nouveau corriger en une
vraie solution $s=s'+s''$, \'equivalente \`a $s_0$ en $p_0$, et $L^2$ \`a
l'infini. L'ambigu\"\i t\'e dans la construction consiste en les solutions
$s$ de $Ps=0$, croissant moins vite en $p_0$ que $s_0$, donc
satisfaisant $s=O(r^{-k})$.

\`A l'infini, l'op\'erateur $P$ a la forme suivante : dans une
trivialisation le long de g\'eod\'esiques aboutissant \`a l'infini, les solutions
de $Ps=0$ ont un comportement gouvern\'e par l'op\'erateur indiciel,
$$ \Ind P=-(x\partial_x)^2-H x\partial_x+P_0,$$
o\`u $H=3$ est la trace de la seconde forme fondamentale, et $P_0$ est un
op\'erateur d'ordre 0 agissant sur $E|_{\partial M_0}$. Une valeur propre $\mu$ de
$P_0$, d'espace propre $E_\mu$, donne un poids critique
\begin{equation}
\delta^\pm=\frac12(H\pm\sqrt{H^2+4\mu}),\label{eq:88}
\end{equation}
ce qui signifie que $x^\delta e$
est une solution de l'op\'erateur indiciel si $e\in E_\mu$. Soit $\delta_0^+>\frac H2$
le plus petit poids critique $L^2$, correspondant \`a la valeur propre $\mu_0$,
alors la solution $s$ fabriqu\'ee plus haut satisfait
$$ s \sim \sigma_\infty x^{\delta_0^+}, \quad \sigma_\infty\in E_{\mu_0}. $$
Ce proc\'ed\'e d\'efinit une application lin\'eaire
\begin{equation}
  \label{eq:89}
  S_k^\infty : E_{p_0}\otimes F_k \longrightarrow \Gamma(\partial M_0,E_{\mu_0}), \quad \sigma_k \longmapsto \sigma_\infty ,
\end{equation}
d\'efini \`a l'image pr\`es des $S_\ell^\infty$ pour $\ell<k$.

D\'ecomposons maintenant le produit tensoriel $E_{p_0}\otimes F_k$ sous
l'action du groupe $\SO(4)$ dans les cas que nous consid\'erons :
\begin{itemize}
\item pour $E=TM_0$, on a $k_0=1$ et 
  \begin{equation}
 E_{p_0}\otimes F_1 = S_+S_-\otimes S_+S_- = S_+^2S_-^2\oplus S_+^2\oplus S_-^2\oplus\setR ;\label{eq:90}
\end{equation}
\item pour $E=\Sym_0^2TM_0$, on a $k_0=0$ et
  \begin{equation}
  \begin{aligned}
    E_{p_0}\otimes F_0 &= S_+^2S_-^2 ,\\
    E_{p_0}\otimes F_2 &= S_+^2S_-^2\otimes S_+^2S_-^2 = (S_+^4\oplus S_+^2\oplus\setR)(S_-^4\oplus S_-^2\oplus\setR).
  \end{aligned}\label{eq:91}
\end{equation}

\end{itemize}

Le lemme suivant indique des cas d'injectivit\'e de l'op\'erateur $S_k^\infty$.
\begin{lemm}\label{lem:S-inj}
  Supposons $M_0$ asymptotiquement hyperbolique, non d\'eg\'en\'er\'ee, alors :
  \begin{enumerate}
  \item pour $P=B\delta^*$ agissant sur $TM_0$, l'op\'erateur $S_1^\infty$
    restreint \`a $S_+^2S_-^2\oplus S_+^2\oplus S_-^2$ est injectif ;
  \item pour $P=\frac12(\nabla^*\nabla-\rR)$ agissant sur $\Sym_0^2T^*M_0$ :
    \begin{itemize}
    \item l'op\'erateur $S_0^\infty$ est injectif ;
    \item l'op\'erateur $S_2^\infty$ est injectif sur $S_+^4S_-^4\oplus S_+^4(S_-^2\oplus\setR)\oplus(S_+^2\oplus\setR)S_-^4$.
    \end{itemize} 
  \end{enumerate}
\end{lemm}
\begin{proof}
  Commen\c cons par le premier cas. L'op\'erateur
  $P=B\delta^*=\frac12(\nabla^*\nabla-\Ric)=\frac12(\nabla^*\nabla+3)$ a 2 poids critiques \`a l'infini :
  \begin{itemize}\item 
    $\delta_0^+=4$, avec espace propre les champs de vecteurs tangents au
    bord (les sections de $T\partial M_0$)\item
    $\delta_1^+>4$, avec espace propre les champs de vecteurs normaux au bord.
  \end{itemize}
  Les solutions $L^2$ de $Ps=0$ ont un d\'eveloppement formel en les
  puissances de $x$, d\'etermin\'e par les deux termes ind\'etermin\'es
  correspondant aux deux poids critiques. Observons \'egalement qu'on a
  un proc\'ed\'e de fabrication de solutions de $Ps=0$ : identifiant
  $TM_0$ avec $\Omega^1M_0$, et compte tenu de la formule de Bochner
  satisfaite par le laplacien de Hodge-De Rham sur $\Omega^1$,
$$ \Delta = \nabla^*\nabla + \Ric, $$
on voit que $P=\frac12(\Delta+6)$. On d\'eduit que si $f$ est une fonction
telle que $(\Delta+6)f=0$, alors $df$ satisfait $Pdf=0$. En particulier le
poids critique de $\Delta+6$ sur les fonctions est \'egal \`a $\delta_1^+$.

On comprend alors pourquoi le facteur $\setR$ dans (\ref{eq:90}) a disparu
dans l'\'enonc\'e du lemme : en effet, il existe une solution de
l'\'equation $(\Delta+6)f=0$, telle que $f\sim \frac1{r^2}$ en $p_0$, et $f$ est
$L^2$ \`a l'infini. Alors $P(df)=0$ donc $df$ fournit un \'el\'ement du
noyau de $S_1^\infty$.

En fait, cela repr\'esente la seule possibilit\'e de non injectivit\'e de
$S_1^\infty$. En effet, si une solution $s$ de $Ps=0$ a un premier terme
ind\'etermin\'e $\sigma_\infty$ qui est nul, alors le d\'eveloppement formel de $s$ \`a
l'infini est de la forme $df$ pour le d\'eveloppement formel d'une
fonction $f$ telle que $(\Delta+6)f=0$.  Il en r\'esulte que le d\'eveloppement
formel de $ds$ \`a l'infini est identifiquement nul. Mais si $s$
satisfait $Ps=0$, alors $ds$ satisfait $(\Delta+6)ds=0$ aussi, et, par
unique continuation \cite{Maz91c}, on d\'eduit que $ds$ est
identiquement nulle sur $M_0$. En particulier, au voisinage de $p_0$,
on a $s=df$ avec $f=O(\frac1{r^2})$. Cela n'est possible que si le
terme principal $\frac{\sigma_1}{r^3}$ de $s$ en $p_0$ satisfait en r\'ealit\'e
que $\sigma_1$ soit dans la composante $\setR$ de la d\'ecomposition
(\ref{eq:90}).

Passons maintenant au second cas, de l'op\'erateur $P$ agissant sur le
fibr\'e $\Sym_0^2T^*M_0$. Il y a trois poids critiques \`a
l'infini :
\begin{itemize}
\item $\delta_0^+=3$, avec espace propre $\Sym_0^2T^*\partial M_0$ ;
\item $\delta_1^+=4$, avec espace propre s'identifiant \`a $T\partial M_0$ (par
  produit int\'erieur avec le vecteur normal);
\item $\delta_2^+>4$, avec espace propre $(3dx^2-\gamma)/x^2$.
\end{itemize}
Comme on a vu, si un champ de vecteurs satisfait $B\delta^*X=0$, alors
$s=\delta^*X$ satisfait $Ps=0$. Cela explique les poids $\delta_1^+$ et
$\delta_2^+$, qui co\"\i ncident avec les deux poids critiques de $B\delta^*$.

Analysons maintenant les op\'erateurs $S_0^\infty$ et $S_2^\infty$ : si une solution
$s$ de $Ps=0$ a son premier terme ind\'etermin\'e nul, $\sigma_\infty=0$, alors son
d\'eveloppement \`a l'infini est celui de $\delta^*X$, pour $X$ un champ de vecteurs
satisfaisant $B\delta^*X=0$. En fait, le th\'eor\`eme d'unique continuation
\cite{Biq08}, voir aussi \cite{AndHer08}, indique que, dans un voisinage de
l'infini, on a exactement $s=\delta^*X$. Comme cette \'equation contr\^ole
compl\`etement $X$ et ses d\'eriv\'ees en fonction de celles de $s$, le champ de
vecteurs $X$ peut \^etre continu\'e analytiquement le long de n'importe quel
chemin. Il en r\'esulte que, pr\`es de $p_0$, on a aussi $s=\delta^*X$ avec
$B\delta^*X=0$ (mais, suivant la topologie, ce n'est pas forc\'ement vrai
globalement).

Dans le cas o\`u $s\sim \frac{\sigma_0}{r^2}$, il faut que $X\sim \frac1r$, et un
tel $X$ satisfaisant $B\delta^*X=0$ n'existe pas. D'o\`u l'injectivit\'e de $S_0^\infty$.

Dans le cas o\`u $s\sim \frac{\sigma_2}{r^4}$, il faut que $X\sim \frac{X_1}{r^3}$,
ce qui est possible. Alors les comportements asymptotiques possibles
de $X$ en $p_0$ sont donn\'es par (\ref{eq:90}). Le terme dominant de
$s$ est obtenu alg\'ebriquement \`a partir de celui de $X$ : en effa\c cant
dans la seconde d\'ecomposition de (\ref{eq:91}) les repr\'esentations de
$\SO(4)$ qui apparaissent dans (\ref{eq:90}), on d\'eduit que ces
comportements asymptotiques de $s$ ne peuvent pas \^etre obtenus comme
$\delta^*X$ et par cons\'equent $S_2^\infty$ est injectif sur ces repr\'esentations.
\end{proof}

\section{Extension de germes sur l'orbifold}
\label{sec:extension-de-germes}

L'existence de petites valeurs propres de l'op\'erateur $P_t$ va \^etre
compens\'ee par des d\'eformations de la m\'etrique, modifiant l'infini conforme
de $g_0$. Par la proposition \ref{prop:ker-P-EH}, l'espace des obstructions
de l'\'equation d'Einstein pour la m\'etrique de Eguchi-Hanson peut \^etre tu\'e
par des tenseurs sym\'etriques sur Eguchi-Hanson, \`a croissance quadratique,
donc asymptotes \`a $\sigma_2r^2$ avec $\sigma_2\in S_+^2S_-^2\otimes S_+^2S_-^2$ comme dans
(\ref{eq:91}). Certains de ces germes correspondent \`a l'action
infinit\'esimale de $\SO(4)$ sur Eguchi-Hanson, donc \`a l'infini \`a $\delta^*X$ avec
$X\sim X_1r$.

Dans le recollement, l'infini de Eguchi-Hanson est vu, apr\`es homoth\'etie,
comme le voisinage de $p_0$. Un ingr\'edient essentiel sera de prolonger
$\sigma_2r^2$ (ou $X_1r$ dans le cas d'un champ de vecteurs), vu comme un germe
sur l'orbifold $M_0$ en $p_0$, en une solution globale $s$ de $Ps=0$ sur
$M_0$. \'Evidemment, un tel $s$ ne saurait \^etre $L^2$ au vu de la non
d\'eg\'en\'erescence de $M_0$. Il y a un plus petit poids critique $\delta_0^-$ non
$L^2$, voir (\ref{eq:88}), plus pr\'ecis\'ement
$$ \delta_0^- =
\begin{cases}
  -1 & \text{ pour }B\delta^*\text{ agissant sur }TM_0, \\
  0 & \text{ pour }P   \text{ agissant sur }\Sym_0^2T^*M_0.
\end{cases}
$$
Le but de cette section est le lemme suivant, dans lequel on se rappellera
les d\'ecompositions (\ref{eq:90}) et (\ref{eq:91}) : en particulier, on
notera $(E_{p_0}\otimes F_2)_0$ le sous-espace de $(S_+^2S_-^2)^{\otimes2}$
orthogonal au facteur trivial $\setR$, c'est-\`a-dire orthogonal \`a la partie
sans trace du 2-tenseur $r^2dr^2$.
\begin{lemm}\label{lem:extension-germes}
  Supposons $M_0$ asymptotiquement hyperbolique, non d\'eg\'en\'er\'ee. Alors
  \begin{enumerate}
  \item pour tout $X_1\in S_+^2S_-^2\oplus S_+^2\oplus S_-^2$, il existe un champ de
    vecteurs $X$ sur $M_0$, satisfaisant $B\delta^*X=0$, tel qu'en $p_0$ on ait
    $X\sim X_1r$ et \`a l'infini on ait $X\sim X_\infty x^{-1}$, o\`u $X_\infty$ est un champ de
    vecteurs sur $\partial M_0$;
  \item pour tout $\sigma_2\in (E_{p_0}\otimes F_2)_0$, il
    existe une solution $s$ de $Ps=0$, tel qu'en $p_0$ on ait $s\sim \sigma_2r^2$
    et \`a l'infini $s\sim s_\infty$, o\`u $s_\infty$ est un 2-tenseur sym\'etrique sans trace
    sur $\partial M_0$.
  \end{enumerate}
\end{lemm}

Ce lemme permet la d\'efinition suivante.
\begin{defi}
  On notera $S_2^\infty$ l'op\'erateur d\'efini par le (2) du lemme
  \ref{lem:extension-germes}, qui \`a un germe $\sigma_2\in (E_{p_0}\otimes F_2)_0$ associe
  le 2-tenseur sym\'etrique $s_\infty$ sur $\partial M_0$.
\end{defi}

Avant de nous attaquer au lemme, nous montrons que l'application entre
valeurs \`a l'infini de solutions de l'\'equation $Ps=0$, et valeurs en
$p_0$, est en un certain sens isom\'etrique. Plus pr\'ecis\'ement, supposons
donn\'ees deux solutions $s$ et $t$ de l'\'equation $Ps=0$, avec les
comportements asymptotiques \guillemotleft{} duaux \guillemotright{} suivants :
\begin{itemize}
\item au point $p_0$, on a $s\sim \sigma_kr^{-2-k}$ et $t\sim \tau_kr^k$ ;
\item \`a l'infini, on a $s\sim \sigma_\infty x^{\delta_0^+}$ et $t\sim \tau_\infty x^{\delta_0^-}$.
\end{itemize}
Puisque l'op\'erateur $P$ est du type $\nabla^*\nabla+\cR$, un calcul imm\'ediat par
int\'egration par parties fournit
\begin{align*}
  \int_{M_0}\langle Ps,t\rangle\vol_{M_0}
&=\int_{M_0}\langle s,Pt\rangle\vol_{M_0}+\lim \int_{\partial M_0}\big(\langle-\nabla_{\vec n}s,t\rangle+\langle s,\nabla_{\vec n}t\rangle\big)\vol_{\partial M_0}\\
&=\int_{M_0}\langle s,Pt\rangle\vol_{M_0}-(2+2k)(\sigma_k,\tau_k)+(\delta_0^+-\delta_0^-)(\sigma_\infty,\tau_\infty),
\end{align*}
et il en r\'esulte la formule, si $Ps=Pt=0$,
\begin{equation}
  \label{eq:92}
  (2+2k) (\sigma_k,\tau_k) = (\delta_0^+-\delta_0^-) (\sigma_\infty,\tau_\infty).
\end{equation}

\begin{proof}[D\'emonstration du lemme \ref{lem:extension-germes}]
  D\'emontrons la premi\`ere partie du lemme : consid\'erons l'application
$$ S_\infty^1 : \Gamma(T\partial M_0) \longrightarrow T_{p_0}M_0\otimes F_1 , $$
qui \`a un champ de vecteurs $X_\infty$ sur $\partial M_0$ associe le germe $X_1$ en
$p_0$ de la solution du probl\`eme
$$ B\delta^*X=0, \quad X \sim \frac{X_1}x \text{ quand }x\to0.$$

  Soit une fonction $f$ solution de l'\'equation
  \begin{equation}\label{eq:93}
 (\Delta+6)f=0, \quad f\sim \frac1{r^2} \text{ en }p_0,
\end{equation}
  et $f$ est $L^2$ \`a l'infini de $M_0$. Comme vu dans la d\'emonstration
  du lemme \ref{lem:S-inj}, on a alors $f=O(x^{\delta_2^+})$ pour un
  $\delta_2^+>4$. Appliquant la formule (\ref{eq:92}) \`a $X$ et \`a $df$, on
  d\'eduit que $X_1$ est orthogonal \`a la composante $\setR$ de $T_{p_0}M_0\otimes
  F_1$. Ainsi, en r\'ealit\'e, au vu de la d\'ecomposition (\ref{eq:90}),
  $$ S_\infty^1 : \Gamma(T\partial M_0) \longrightarrow S_+^2S_-^2\oplus S_+^2\oplus S_-^2 . $$

  Supposons que cet op\'erateur $S_\infty^1$ ne soit pas surjectif. Par
  cons\'equent, il existe $Y_1\in S_+^2S_-^2\oplus S_+^2\oplus S_-^2$, tel que pour
  tout $X_\infty$, on ait
  \begin{equation}
 (Y_1,S_\infty^1(X_\infty))=0.\label{eq:94}
  \end{equation}
  D'apr\`es le lemme \ref{lem:S-inj}, il existe $Y$ champ de vecteurs
  sur $M_0$, solution de
$$ B\delta^*Y=0, \quad Y\sim \frac{Y_1}{r^3} \text{ en }p_0, $$
  et, \`a l'infini, $Y\sim Y_\infty x^4$ avec $Y_\infty\neq 0$. D'apr\`es (\ref{eq:92}) et
  (\ref{eq:94}), on d\'eduit
$$ (Y_\infty,X_\infty)=0 $$
  pour tout $X_\infty$, ce qui implique $Y_\infty=0$. On obtient une
  contradiction. Donc $S_\infty^1$ est en r\'ealit\'e surjectif, ce qui
  d\'emontre la premi\`ere assertion du lemme.

  Passons \`a pr\'esent \`a la seconde assertion, et consid\'erons donc
  l'op\'erateur $P=\frac12\nabla^*\nabla-\rR$. Par le lemme \ref{lem:S-inj},
  nous avons un op\'erateur injectif
  $S_0^\infty:\Sym_0^2T_{p_0}M_0\to\Gamma(\Sym_0^2 \partial M_0)$, dont nous noterons
  l'image $E_0$. Si $s_\infty\in E_0^\perp$, alors la solution $s$ de $Ps=0$
  telle que $s\sim s_\infty$, d'apr\`es la formule (\ref{eq:92}), est nulle en
  $p_0$, donc $s\sim s_2r^2$ avec $s_2\in \Sym_0^2T_{p_0}M_0\otimes F_2$. On a
  ainsi d\'efini une application
$$ S_\infty^2:E_0^\perp \longrightarrow \Sym_0^2T_{p_0}M_0\otimes F_2 . $$
  Utilisons (\ref{eq:91}) pour d\'ecomposer le but de cette application comme
$$ \Sym_0^2T_{p_0}M_0\otimes F_2 = \setR\oplus R_1\oplus R_2, $$
  avec
  \begin{align*}
    R_1&=T_{p_0}M_0\otimes F_1=S_+^2S_-^2\oplus S_+^2\oplus S_-^2,\\
    R_2&=S_+^4S_-^4\oplus S_+^4(S_-^2\oplus\setR)\oplus(S_+^2\oplus\setR)S_-^4.
  \end{align*}
  Commen\c cons par observer que l'image de $S_\infty^2$ est orthogonale au
  facteur $\setR$ : cela se voit en int\'egrant la solution $s$ contre
  $\delta^*df$, o\`u $f$ est la fonction construite par (\ref{eq:93}). Nous
  restreignons donc $S_\infty^2$ en une application
$$ S_\infty^2:E_0^\perp \longrightarrow R_1\oplus R_2 , $$
  dont nous voulons montrer la surjectivit\'e. Il r\'esulte d\'ej\`a de la
  premi\`ere partie du lemme que l'image de $S_\infty^2$ contient
  $R_1$. Supposons donc qu'existe $s_2\in R_2$, orthogonal \`a l'image de
  $S_\infty^2$. Par le lemme \ref{lem:S-inj}, il existe une solution $s$ de
  $Ps=0$, telle qu'en $p_0$,
  $$ s\sim \frac{s_2}{r^4}, $$
  et \`a l'infini,
  $$ s\sim s_\infty x^3, \quad s_\infty\neq 0. $$
  Par (\ref{eq:92}), il faut que $s_\infty\in E_0$. Mais cela implique
  l'existence d'une solution $t$ de $Pt=0$, tel qu'\`a l'infini $t\sim
  s_\infty x^3$ et en $p_0$ on ait seulement $t=O(r^{-2})$. Par continuation
  unique \`a nouveau \cite{Biq08}, il faut que, pr\`es de l'infini, on ait
  $s=t+\delta^*X$. Le champ de vecteurs $X$ peut \^etre continu\'e analytiquement
  le long de chemins, et donc pr\`es de $p_0$ on obtient aussi l'existence de
  $X$ tel que $s=t+\delta^*X$. Mais cela est incompatible avec $s_2\in R_2$,
  qui est justement l'orthogonal des valeurs induites par les champs
  de vecteurs.

  Il en r\'esulte que $S_\infty^2$ est surjective, et le lemme est d\'emontr\'e.
\end{proof}

\section{Construction des m\'etriques d'Einstein}
\label{sec:metriques-E}

Les tenseurs sans trace $k_j$ de la proposition \ref{prop:ker-P-EH} ont une
asymptotique \`a l'infini d'ordre 2,
$$ k_j \sim \eta_j, $$
o\`u $\eta_j\in E_{p_0}\otimes F_2$. Il est \'evident pour $\eta_2$, $\eta_3$, et il se v\'erifie
par le calcul pour $\eta_1$, qu'en r\'ealit\'e $\eta_j\in (E_{p_0}\otimes F_2)_0$.
\begin{lemm}\label{lem:d-lambda}
  Pour un infini conforme $\gamma$, soit $\lambda_j(\varphi,\gamma)$ d\'efinies par
  (\ref{eq:24}). Alors
  $$ \frac{\partial\lambda_\ell}{\partial\gamma}\big|_{(1,\gamma_0)}\big(S_2^\infty(\eta_j)\big) = \delta_{j\ell}.$$
\end{lemm}
\begin{proof}
  La variation $S_2^\infty(\eta_j)$ de $\gamma$ a \'et\'e construite de sorte que la
  variation correspondante des termes quadratiques de $g_0(\gamma)$ en $p_0$
  soit \'egale aux termes quadratiques du tenseur $k_j$ \`a l'infini de
  Eguchi-Hanson. Donc la solution $h$ du probl\`eme (\ref{eq:23}) est
  modifi\'ee par $k_j$, tel que $d_\eh\Ric(k_j)=o_j$, et le lemme en r\'esulte
  imm\'ediatement.
\end{proof}

Par ailleurs, examinons l'influence du param\`etre de jauge $\varphi\in
\Sp_1/U_1$ sur les coefficients $\lambda_j$. Le r\'esultat d\'epend de la partie
$\bR_{g_0}^+(p_0)$ de l'op\'erateur de courbure en $p_0$ agissant sur
$\Omega^2_+(p_0)$. Supposons que $\det \bR_{g_0}^+(p_0)=0$, ce qui signifie
que la condition (\ref{eq:26}) est satisfaite, apr\`es action \'eventuelle
de la jauge.  On peut donc supposer qu'au d\'epart, dans la base
standard $(I_1=dx^1\land dx^2+dx^3\land dx^4,I_2=dx^1\land dx^3-dx^2\land
dx^4,I_3=dx^1\land dx^4+dx^2\land dx^3)$ de $\Omega^2_+(p_0)$, l'op\'erateur de
courbure a la forme diagonale
\begin{equation}
 \bR_{g_0}^+(p_0) =
\begin{pmatrix}
  0 & & \\ & a_2 & \\ & & a_3
\end{pmatrix}\label{eq:95}
\end{equation}
avec par exemple $a_3\neq 0$ puisque $\tr(\bR^+)=\frac14\Scal \neq 0$. On
dira que $\bR_{g_0}^+(p_0)$ est \emph{non d\'eg\'en\'er\'e} si on a aussi
$a_2\neq 0$. On peut repr\'esenter un \'el\'ement de $\mathfrak{sp_1}/\mathfrak
u_1$ comme
$$ \xi = \begin{pmatrix}
  0 & -\xi_2 & -\xi_3 \\ \xi_2 & 0 & 0 \\ \xi_3 & 0 & 0
       \end{pmatrix}.$$
Sachant que l'action de $\Sp_1$ sur $\bR_{g_0}^+(p_0)$ est par
conjugaison, et que $(\lambda_1,\lambda_2,\lambda_3)$ sont les coefficients de
$\bR_{g_0}^+(p_0)(I_1)$, on d\'eduit imm\'ediatement :
\begin{lemm}\label{lem:d-lambda-jauge}
  Si $\xi\in \mathfrak{sp}_1/\mathfrak u_1$, alors 
  \begin{equation}
    \frac{\partial\lambda_1}{\partial\varphi}\big|_{(1,\gamma_0)}=0,\quad
    \frac{\partial\lambda_j}{\partial\varphi}\big|_{(1,\gamma_0)}(\xi) = a_j\xi_j \text{ pour }j=2,3.\label{eq:96}
  \end{equation}\qed
\end{lemm}

Soient $\gamma(u_1,u_2)$ une famille d'infinis conformes tels que
\begin{equation}
  \label{eq:97}
  \gamma(0)=\gamma_0, \quad \frac{\partial\gamma}{\partial u_j}\big|_{u=0} = S_2^\infty(\eta_j) .
\end{equation}
On notera aussi
$$ \gamma(u)=\gamma(u,0) . $$
On d\'eduit imm\'ediatement du d\'eveloppement (\ref{eq:64}), et des lemmes
\ref{lem:d-lambda} et \ref{lem:d-lambda-jauge}, la construction des
m\'etriques d'Einstein :
\begin{theo}\label{theo:constr-E}
  Si $(M_0,g_0)$, non d\'eg\'en\'er\'e, satisfait la condition de courbure
  $\det\bR_{g_0}^+(p_0)=0$, alors :
  \begin{itemize}
  \item si $\bR_{g_0}^+(p_0)$ est non d\'eg\'en\'er\'e (une seule valeur
    propre de $\det\bR_{g_0}^+(p_0)$ s'annule), alors pour $t>0$
    suffisamment petit, il existe un r\'eel $u(t)$ et une jauge $\varphi(t)\in
    \Sp_1/U_1$, tels que $u(t)=O(t^{\frac12-\delta})$ pour tout $\delta>0$, et, pour
    $j=1,2,3$, 
    $$ \lambda_j\big(t,\varphi(t),\gamma(u(t))\big)=0, $$
    c'est-\`a-dire $g_{t,\varphi(t),\gamma(u(t))}$ est une m\'etrique d'Einstein ;
  \item en g\'en\'eral, pour $t>0$ suffisamment petit, il existe des r\'eels
    $u_1(t), u_2(t)$ et une jauge $\varphi(t)\in \Sp_1/U_1$, tels que
    $u_j(t)=O(t^{\frac12-\delta})$ pour tout $\delta>0$, et, pour $j=1,2,3$,
    $$ \lambda_j\big(t,\varphi(t),\gamma(u_1(t),u_2(t))\big)=0. $$
  \end{itemize}
  \qed
\end{theo}
\begin{rema}
  La seconde partie du th\'eor\`eme est toujours vraie : le lien avec la
  premi\`ere provient du fait suivant : dans le cas non d\'eg\'en\'er\'e, on
  peut choisir $\gamma(u_1,u_2)$ diff\'eomorphe \`a $\gamma(u_1)$, ce qui est une
  mani\`ere d\'etourn\'ee de changer la jauge \`a l'origine (on modifie la
  m\'etrique $g_0$ par une perturbation globale de la jauge au lieu de
  modifier le recollement) ; c'est une cons\'equence du (1) du lemme
  \ref{lem:extension-germes}. C'est aussi pourquoi, dans la forme
  (\ref{eq:95}), le $\lambda_3$ pourrait \^etre compens\'e par une d\'eformation
  de l'infini conforme par des diff\'eomorphismes, \'equivalente au
  changement de jauge au recollement.
\end{rema}
\begin{rema}\label{rem:non-unicite:E}
  Les m\'etriques fabriqu\'ees dans la proposition \ref{prop:resol-O}
  d\'ependent d'un param\`etre $\upsilon\in \setR^3$. Le th\'eor\`eme a \'et\'e \'enonc\'e en
  utilisant les m\'etriques construites pour $\upsilon=0$, mais toute autre
  valeur (petite) de $\upsilon$ est possible. Si on fixe la famille $\gamma(t)$ de
  d\'eformations autoris\'ees de l'infini conforme, on peut n\'eanmoins
  s'attendre \`a un r\'esultat d'unicit\'e de la m\'etrique d'Einstein, \`a
  infini conforme fix\'e. Mais il semble impossible de contr\^oler le
  noyau de la lin\'earisation $P_{g_t}$ apr\`es d\'eformation, qui garde des
  petites valeurs propres issues du noyau de $P_\eh$ (voir
  \S~\ref{sec:petit-vale-propr}). N\'eanmoins, on verra en
  \S~\ref{sec:transversalite} que g\'en\'eriquement, tel est bien le cas :
  les variables $\upsilon$ n'apportent pas de param\`etre suppl\'ementaire, et
  les m\'etriques d'Einstein construites sont d\'etermin\'ees par leur
  infini conforme (th\'eor\`eme \ref{theo:transversalite}).
\end{rema}

\section{Petites valeurs propres de la lin\'earisation}
\label{sec:petit-vale-propr}

Pour obtenir une meilleure approximation de la m\'etrique d'Einstein, il
faut raffiner substantiellement la construction faite jusqu'ici. La
premi\`ere \'etape consiste \`a remplacer les obstructions $\chi_to_i$,
obtenues par troncation na\"\i ve des obstructions de Eguchi-Hanson, par
de meilleures approximations des vecteurs propres associ\'es aux petites
valeurs propres de l'op\'erateur $P_t$ sur $M$.

Sur Eguchi-Hanson, nous disposons donc d'un noyau de $P$, engendr\'e par
$o_1$, $o_2$ et $o_3$, avec comportement asymptotique (\ref{eq:27}),
que nous r\'e\'ecrivons
\begin{equation}
 o_j \sim \frac{\eta_j}{r^6},\label{eq:98}
\end{equation}
o\`u les $\eta_j$ sont des polyn\^omes homog\`enes de degr\'e 2, donc a priori
$\eta_j\in (S_+^2S_-^2)^{\otimes2}$. Il r\'esulte de la formule (\ref{eq:27}), voir
aussi la suite de la section \ref{sec:reform-des-obstr}, qu'en r\'ealit\'e
on a 
\begin{equation}
\eta_j\in S_+^4\subset (S_+^2S_-^2)^{\otimes2}.\label{eq:99}
\end{equation}

Comme vu en \S\ref{sec:solutions-orbifold}, il existe sur $M_0$ des
solutions $\bar o_j$ de l'\'equation $P\bar o_j=0$, qui soient $L^2$ \`a
l'infini et aient pr\`es de $p_0$ le comportement
$$ \bar o_j \sim \frac{\eta_j}{r^6}. $$
Compte tenu de ce comportement asymptotique, on a $B\bar
o_j=O(r^{-3})$, donc on peut trouver un champ de vecteurs
$X_j=O(r^{-1})$, tel que $B\delta^*X_j=-B\bar o_j$. En rempla\c cant
$\bar o_j$ par $\bar o_j+\delta^*X_j$, on peut ainsi supposer
$$ B\bar o_j=0. $$
Nous d\'efinissons alors le recollement $o_{j,t}$ sur $M$, par
\begin{equation}
  \label{eq:100}
  o_{j,t} = (1-\chi_t)t\bar o_j + \chi_t o_j,
\end{equation}
o\`u $\chi_t$ est la fonction de coupure d\'efinie par (\ref{eq:40}). Le
facteur $t$ est le facteur n\'ecessaire pour faire co\"\i ncider les
comportements asymptotiques de $\bar o_j$ en $p_0$ et de $o_j$ \`a
l'infini apr\`es application de l'homoth\'etie $s_t$, de rapport $\sqrt{t}$.

Par d\'efinition, $B_{g_t}o_{j,t}$ et $P_to_{j,t}$ ont leur support inclus
dans $\EH^t$. Les estimations suivantes am\'eliorent d'un facteur
$t^{\frac12}$ les estimations (\ref{eq:75}) et (\ref{eq:76}) :
\begin{lemm}\label{lem:Ph}
  On a sur $\EH^t$ les estimations, pour tout entier naturel $k$,
  \begin{align*}
    |\nabla^k B_{g_t}o_{j,t}|_{\frac{g_t}t} &\leq c_k t^{-1} \rho^{-3-k},\\
    |\nabla^k P_{\frac{g_t}t} o_{j,t}|_{\frac{g_t}t} &\leq c_k t \rho^{-4-k} .
  \end{align*}
\end{lemm}
\begin{proof}
  Les termes d'ordre $r^{-4}$ de $o_j$ et $t\bar o_j$ co\"\i ncident, donc la
  diff\'erence satisfait $|o_j-t\bar o_j|_\eh \leq c \rho^{-6}$ sur
  $\EH^t_>$. Il en r\'esulte qu'on peut am\'eliorer (\ref{eq:76}) en
  $$ |B_{g_t}(\chi_to_{j,t})|_{\frac{g_t}t}
    \leq c t^{-2} \big( \frac{1_{\EH^t_>}}{\rho^7}+\frac t{\rho^3} \big).$$
  La premi\`ere estimation s'en d\'eduit. L'am\'elioration de (\ref{eq:75}) est
  similaire.
\end{proof}

Ce lemme montre que la lin\'earisation $P_t$ a des petites valeurs
propres : en effet, quand $t\to0$, on a $o_{j,t}\to o_j$ sur $\EH$, et plus
pr\'ecis\'ement $$\|o_{j,t} \|_{L^2(g_t)}\to\|o_j\|_{L^2(\eh)}$$ alors que
\begin{equation*}
  \| P_to_{j,t} \|_{L^2(g_t)} = \| P_{\frac{g_t}t}t^{-1}o_{j,t} \|_{L^2(\frac{g_t}t)}
    = O(1).
\end{equation*}
En revanche, la troncation $\chi_to_j$, utilis\'ee auparavant, ne permet
pas d'obtenir ces valeurs propres born\'ees.

L'application importante pour nous est l'am\'elioration suivante des
estimations (\ref{eq:72}) et (\ref{eq:73}) :
\begin{coro}\label{cor:meilleures-estimations}
  On a les estimations
  \begin{align*}
    \big| (P_th,o_{j,t})_{g_t} \big| &\leq c t^{1-\frac{\delta_0}2} \|h\|_{C^{2,\alpha}_{\delta_0,\delta_\infty;t}},\\
    \big| (\pi_{g_t}o_{j,t},o_{k,t}) \big| &\leq (o_j,o_k)_\eh + O(t^{1-\delta})
      \text{ pour tout }\delta>0.
  \end{align*}
\end{coro}
\begin{proof}
  Comme dans (\ref{eq:74}), on utilise l'int\'egration par parties
  $$\big(P_th,o_{j,t}\big)_{g_t}=\big(\tfrac ht,P_{\frac{g_t}t}o_{j,t}\big)_{\frac{g_t}t},$$
  et la premi\`ere estimation r\'esulte alors imm\'ediatement du lemme
  \ref{lem:Ph}.  Celui-ci nous fournit aussi, \`a la place de
  (\ref{eq:77}),
  $$ \|B_{g_t}o_{j,t}\|_{C^{k,\alpha}_{\delta+2}} \leq c t^{\frac12+\frac\delta4} $$
  qui am\'eliore automatiquement d'un facteur $t^{\frac12}$ l'estimation
  (\ref{eq:73}).
\end{proof}
\begin{rema}
  On d\'eduit du corollaire que la proposition \ref{prop:resol-O} reste
  valable en rempla\c cant $\chi_to_j$ par $o_{j,t}$. Cela am\'eliore
  l'estimation (\ref{eq:64}) en
  $$ \lambda_j(t)=t\lambda_j + O(t^{2-\frac{\delta_0}2}) .$$
  Le poids $\frac{\delta_0}2$ restant vient de (\ref{eq:84}). Ce
  d\'eveloppement est insuffisant pour obtenir celui des solutions, qui
  requiert le terme en $t^2$ de $\lambda_j(t)$.
\end{rema}

\section{Asymptotique de la solution}
\label{sec:asymptotique}

La d\'etermination de l'asymptotique de la solution quand $t$ tend vers
$0$ se fait en raffinant la construction des solutions approch\'ees. La
solution approch\'ee $g_t$ a \'et\'e construite en recollant $g_0$ avec une
modification $h_t$ de la m\'etrique de Eguchi-Hanson par un terme
asymptotique aux termes quadratiques de $g_0$ en $p_0$. Une meilleure
solution approch\'ee sera obtenue, d'une part en raffinant $h_t$ pour
obtenir la co\"\i ncidence avec les termes d'ordre 4 de $g_0$, d'autre
part en modifiant $g_0$ pour obtenir la co\"\i ncidence avec les termes en
$r^{-4}$ de la m\'etrique de Eguchi-Hanson.

Nous reprenons la section \ref{sec:le-probl-line}. On dispose sur $\EH^t$
d'une m\'etrique $h_t=\eh+th$, o\`u $h$ est solution du syst\`eme
(\ref{eq:23}). Pour pousser plus loin le d\'eveloppement d'une solution de
$\Ric(g)=t\Lambda g \mod\cO_\eh$, d\'eveloppons formellement
$$ \Ric(\eh+u) = d_\eh\Ric(u) + Q(u) + R(u), $$
o\`u $Q(u)$ contient uniquement les termes quadratiques en $u$ et $R(u)$
les termes d'ordre au moins 3. Alors le terme $t^2h_2$ de $g$ sera
obtenu en r\'esolvant
\begin{equation}
 d_\eh\Ric(h_2) = \Lambda h - Q(h) \mod \cO_\eh.\label{eq:101}
\end{equation}
\`A l'infini, le terme $h_2$ doit \^etre d'ordre 4, et pour recoller avec
la m\'etrique $g_0$ sur $M_0$, on \'ecrit en $p_0$ le d\'eveloppement
$$ g_0 = \euc + H + H_2 + O(r^6),$$
o\`u $H_2$ contient les termes d'ordre 4 de $g_0$ seulement. Alors la
condition \`a l'infini \`a consid\'erer sur $h_2$ est
\begin{equation}
  \label{eq:102}
  h_2 = H_2 + O(\rho^2) .
\end{equation}
Le probl\`eme (\ref{eq:101}) (\ref{eq:102}) est \`a nouveau obstru\'e par les
tenseurs $o_j$, et on r\'esout \`a la place le syst\`eme
\begin{equation}
  \label{eq:103}
  \begin{aligned}
    d_\eh\Ric(h_2) &= \Lambda h - Q(h) + \sum_1^3 \mu_j o_j , \\
    h_2 &= H_2 + O(\rho^2) ,
  \end{aligned}
\end{equation}
o\`u les constantes $\mu_j$ font partie des inconnues.
\begin{lemm}
  Le probl\`eme (\ref{eq:103}) a toujours une solution. Le tenseur
  $$\hh_t=\eh+th+t^2h_2$$ est une m\'etrique sur $\EH^t$, satisfaisant
  \begin{equation}
    \label{eq:104}
    \big| \nabla^k\big(\Ric_{\hh_t}-t\Lambda\hh_t-\sum_1^3(t\lambda_j+t^2\mu_j)o_j \big)\big|_\eh \leq c_k t^3 \rho^{4-k} .
  \end{equation}
  Les constantes $\mu_j$ sont une fonction du 2-jet de la courbure de
  $g_0$ au point $p_0$.
\end{lemm}
\begin{proof}
  La r\'esolution de l'\'equation se fait comme en
  \S~\ref{sec:le-probl-line}. L'estimation est similaire \`a celle du lemme
  \ref{lem:est-ini-Ric} : on est juste all\'e une \'etape plus loin dans la
  r\'esolution de l'\'equation d'Einstein, en compensant par $h_2$ les termes
  quadratiques, donc il ne reste plus dans $\Ric_{\hh_t}-t\Lambda\hh_t$ que les
  obstructions et les termes d'ordre au moins 3, ce qui m\`ene au lemme.
\end{proof}
Si on varie la jauge $\varphi$ et l'infini conforme $\gamma$, on obtient une m\'etrique $\hh_{t,\varphi,\gamma}$ et
des constantes $\lambda_j(\varphi,\gamma)$ et $\mu_j(\varphi,\gamma)$. On a maintenant une m\'etrique
d'Einstein modulo obstructions \`a un ordre \'elev\'e par
(\ref{eq:104}). N\'eanmoins, pour maintenir la qualit\'e de l'approximation dans
le recollement, il faut raffiner aussi $g_0(\gamma(t))$ de la mani\`ere
suivante. \'Ecrivons, par (\ref{eq:4}),
$$ \eh = \euc + \frac \sigma{\rho^6} + O(\frac1{\rho^6}), $$
o\`u $\sigma$ est un terme homog\`ene de degr\'e 2. Comme dans la section
\ref{sec:solutions-orbifold}, r\'esolvons le probl\`eme
\begin{equation}
  \label{eq:105}
  \begin{aligned}
    P_{g_0}k &= 0, \\
    B_{g_0}k &= 0, \\
    k &\sim \frac \sigma{r^6} \text{ en }p_0,
  \end{aligned}
\end{equation}
et $k$ est $L^2$ \`a l'infini de $M_0$. Alors le tenseur
$$ g_0 + t^2 k$$
est une m\'etrique sur $M_0^t$. Bien s\^ur tout d\'epend ici de l'infini
conforme $\gamma$ ainsi que de la jauge $\varphi$ et on notera explicitement
cette d\'ependance $k(\varphi,\gamma)$.

Nous pouvons maintenant d\'efinir une solution \`a un ordre \'elev\'e de l'\'equation
d'Einstein modulo obstructions :
\begin{equation}
  \label{eq:106}
  \hg_{t,\varphi,\gamma} = t \chi_t \hh_{t,\varphi,\gamma} + (1-\chi_t) (g_0(\gamma)+t^2k(\varphi,\gamma)) .
\end{equation}
Elle satisfait
\begin{lemm}\label{lem:est-meilleure-approx}
  Sur $\EH^t$ on a les estimations
  \begin{align}
    \big| \nabla^k(\Ric-\Lambda)(\hg_t)-\sum_1^3(t\lambda_j+t^2\mu_j)o_{j,t} \big|_\eh &\leq c_k t^3 \rho^{4-k},\label{eq:107}\\
  \intertext{et sur $M_0^t$,}
    \big| \nabla^k(\Ric-\Lambda)(\hg_t)-\sum_1^3(t\lambda_j+t^2\mu_j)o_{j,t} \big|_{g_0} &\leq c_k (r^{4-k}1_{\EH^t_>}+t^4\tilde r^{-10-k}x^6).\label{eq:108}
  \end{align}
\end{lemm}
\begin{proof}
  C'est une cons\'equence imm\'ediate, d'une part de l'estimation
  (\ref{eq:104}), d'autre part du fait que les 2 m\'etriques recoll\'ees, \`a
  savoir $\hat h_{t,\varphi,\gamma}$ et $g_0(\gamma)+t^2k(\varphi,\gamma)$ co\"\i ncident \`a un ordre \'elev\'e
  sur la transition $\EH^t_>$ :
  $$ \big|\hat h_{t,\varphi,\gamma} - (g_0(\gamma)+t^2k(\varphi,\gamma))\big|_\eh \leq c( \frac 1{\rho^6} + t^3 \rho^6 ).
  $$
  On remarquera que les deux termes de droite sont du m\^eme ordre de
  grandeur. L'estimation (\ref{eq:107}) s'en d\'eduit imm\'ediatement. L'autre
  estimation (\ref{eq:108}) sur la partie $M_0^t$ vient du fait que la
  m\'etrique $g_0(\gamma)+t^2k(\varphi,\gamma)$ n'est plus d'Einstein qu'\`a l'ordre 1, et on
  doit donc tenir compte des termes quadratiques provenant de $t^2k(\varphi,\gamma)$.
\end{proof}
Compte tenu, sur les 2-tenseurs, de l'\'egalit\'e
$$|\nabla^k\cdot|_{\frac{g_0}t}=t^{1+\frac k2}|\nabla^k\cdot|_{g_0},$$
la majoration (\ref{eq:107}) et le terme $r^{4-k}$ dans (\ref{eq:108})
donnent la m\^eme estimation sur la zone de transition $\EH^t_>$. Le terme
$t^4x^6\tilde r^{-10-k}$ est provoqu\'e par la correction $t^2k(\gamma)$ de
$g_0(\gamma)$ : il est global sur $M_0$, mais d'ordre plus petit sur $\EH^t_>$ ;
l'annulation $x^6$ pr\`es de $\partial M_0$ vient de $|k| =O(x^3)$.

Notons que sans la correction $t^2k(\gamma)$, les termes d'ordre $O(\rho^{-4})$ ne
se recollent pas, et provoquent une erreur $O(\rho^{-6})\sim O(t^{\frac 32})$
dans le tenseur de Ricci, ce qui serait insuffisant pour (\ref{eq:107}).

Par le lemme \ref{lem:est-meilleure-approx}, on obtient
\begin{equation}
  \label{eq:109}
  \| (\Ric-\Lambda)(\hg_t)-\sum_1^3(t\lambda_j+t^2\mu_j)o_{j,t} \|_{C^\alpha_{\delta_0+2,\delta_\infty;t}} = O(t^{\frac32+\frac{\delta_0}4}) .
\end{equation}
La m\'ethode de r\'esolution de l'\'equation d'Einstein employ\'ee \`a partir de la
solution approch\'ee $g_{t,\varphi,\gamma}$ pour d\'emontrer la proposition
\ref{prop:resol-O} s'applique aussi bien \`a partir de $\hg_t$. On en d\'eduit
imm\'ediatement l'existence et l'unicit\'e d'une m\'etrique $\hat\hg_t$, par
exemple en jauge de Bianchi par rapport \`a $\hg_t$, et telle que
\begin{equation}
  \label{eq:110}
  \Ric(\hat\hg_t)-\Lambda\hat\hg_t = \sum_1^3 \hat \lambda_j(\varphi,\gamma) \pi_{\hat\hg_t}(o_{j,t})
\end{equation}
Cette m\'etrique satisfait
\begin{equation}
  \label{eq:111}
  \|\hat\hg_t-\hg_t\|_{C^{2,\alpha}_{\delta_0,\delta_\infty;t}} = O(t^{\frac32+\frac{\delta_0}4}) .
\end{equation}
En outre, vu les estimations du corollaire
\ref{cor:meilleures-estimations}, on en d\'eduit
\begin{equation}
  \label{eq:112}
  \hat \lambda_j(\varphi,\gamma) = t \lambda_j(\varphi,\gamma) + t^2 \mu_j(\varphi,\gamma) + O(t^{\frac52}) .
\end{equation}

Par le lemme \ref{lem:d-lambda}, on peut maintenant annuler les
coefficients $\hat \lambda_j$ des termes d'obstruction et raffiner ainsi le
th\'eor\`eme \ref{theo:constr-E} :
\begin{theo}\label{theo:constr-E-raffinee}
  Si $(M_0,g_0)$, non d\'eg\'en\'er\'ee, satisfait la condition de courbure
  $\det\bR^+_{g_0}(p_0)=0$, avec jauge ajust\'ee de sorte que les
  $\lambda_j(\gamma_0)=0$, alors il existe une jauge $\varphi(t)\in \Sp_1/U_1$ et une famille
  d'infinis conformes $\gamma(t)$ telle que $\gamma(0)=\gamma_0$, et pour tout $t>0$ assez
  petit, la m\'etrique $\hat\hg_{t,\varphi(t),\gamma(t)}$ soit d'Einstein.  Une telle
  famille $\gamma(t)$ doit satisfaire
  \begin{equation}
    \lambda_j(\varphi(t),\gamma(t)) = - t \mu_j(\varphi(0),\gamma_0) + O(t^{\frac32}).\label{eq:113}
  \end{equation}\qed
\end{theo}
Par exemple, dans le cas o\`u $\bR_{g_0}^+(p_0)$ est non d\'eg\'en\'er\'e,
$\gamma(t)$ peut \^etre choisie dans la famille \`a un param\`etre $\gamma(u)$
utilis\'ee pour le th\'eor\`eme \ref{theo:constr-E}, de sorte que
$$\dot \gamma(0)=-\mu_1(\gamma_0)S_2^\infty(\eta_1).$$

Notons \`a nouveau (remarque \ref{rem:non-unicite:E}) qu'aucune unicit\'e de la
m\'etrique d'Einstein produite n'est prouv\'ee, et qu'en particulier elle
semble \`a premi\`ere vue d\'ependre du param\`etre $\upsilon$ de la proposition
\ref{prop:resol-O}.

La motivation pour obtenir une asymptotique en $t$ beaucoup plus pr\'ecise
est l'analyse, dans le voisinage de $\gamma_0$, des infinis conformes obtenus
comme bords de m\'etriques d'Einstein lisses. Notons $\cC$ l'espace des
m\'etriques conformes sur $\partial M_0$. Consid\'erons l'hypersurface
\begin{equation}
\cC_0=\{ \det \bR_{g_0}^+(p_0)=0 \} \subset\cC .\label{eq:114}
\end{equation}
D'un point de vue technique, pour travailler dans des vari\'et\'es
banachiques, il faut regarder l'infini conforme $\gamma$ dans un espace de
H\"older $C^{k,\alpha}$. Si $\bR_{g_0}^+(p_0)$ est non d\'eg\'en\'er\'e, alors,
d'apr\`es le lemme \ref{lem:d-lambda}, $\cC_0$ est une hypersurface
lisse au point $\gamma_0$. Le d\'eveloppement (\ref{eq:113}) a pour
cons\'equence le r\'esultat suivant :

\begin{theo}\label{theo:infinis-conformes}
  Soit $\gamma_0\in \cC_0$. Si $\bR_{g_0}^+(p_0)$ est non d\'eg\'en\'er\'e et $\mu_1(\gamma_0)\neq
  0$ (par exemple fixons $\mu_1(\gamma_0)>0$), alors dans un petit voisinage de
  $\cC_0=\{\lambda_1=0\}$, les infinis conformes des m\'etriques d'Einstein produites
  par le th\'eor\`eme \ref{theo:constr-E-raffinee} sont exactement les
  m\'etriques conformes v\'erifiant $\lambda_1<0$.\qed
\end{theo}
La m\^eme chose est valable si on varie le param\`etre apparent $\upsilon$.

Au vu de l'argument d'unicit\'e (corollaire \ref{cor:obs-E}), sous les
hypoth\`eses du th\'eor\`eme, toutes les m\'etriques d'Einstein suffisamment
proches d'un recollement na\"\i f $g_{t,\varphi,\gamma}^0$ doivent co\"\i ncider avec une
m\'etrique d'Einstein dans la famille que nous avons construite, et donc
satisfaire $\lambda_1(\gamma)<0$ dans le cas non d\'eg\'en\'er\'e :
\begin{coro}\label{cor:lambda1negatif}
  Sous les m\^emes hypoth\`eses, si une m\'etrique AH d'Einstein $g$ est
  suffisamment proche d'un recollement $g^0_{t,\varphi,\gamma}$ au sens du corollaire
  \ref{cor:obs-E}, alors son infini conforme doit satisfaire $\lambda_1<0$.\qed
\end{coro}
Cela sugg\`ere fortement que l'hypersurface $\cC_0$ est un mur pour la
r\'esolution du probl\`eme de Dirichlet \`a l'infini sur $M$, c'est-\`a-dire que la
r\'esolution n'est possible que d'un seul c\^ot\'e de $\cC_0$. Cela sera
rendu pr\'ecis en \S~\ref{sec:transversalite}.

Enfin, on peut se demander si le c\^ot\'e de $\cC_0$ o\`u l'on peut
d\'esingulariser ne d\'epend que du signe de $\det \bR_{g_0}^+(p_0)$, par
exemple serait toujours $\det \bR_{g_0}^+(p_0)>0$ (comme le sugg\`ere
l'exemple des m\'etriques Ads-Taub-Bolt, voir apr\`es le th\'eor\`eme \ref{theo:transversalite}. L\`a encore, une compr\'ehension fine des termes
quadratiques de l'\'equation d'Einstein serait n\'ecessaire pour trancher.

\section{Transversalit\'e}
\label{sec:transversalite}

Dans cette section, nous donnons des arguments de transversalit\'e
visant \`a pr\'eciser l'espace des modules de m\'etriques d'Einstein AH dans
le voisinage de l'hypersurface $\cC_0\subset\cC$ d\'efinie par (\ref{eq:114}).

Soit $\gamma_0\in \cC_0$, tel que $\mu_1(\gamma_0)\neq 0$, (donc $\cC_0$ est lisse en
$\gamma_0$), alors nous disposons d'une direction $\delta(\gamma_0)=S_2^\infty(\eta_1)$ dans
$\cC$ au point $\gamma_0$, transverse \`a $\cC_0$, telle qu'on puisse
r\'esoudre l'\'equation d'Einstein pour tout infini conforme
\begin{equation}
  \label{eq:115}
  \gamma(t) = \gamma_0 + t \delta(\gamma_0), \quad t>0 \text{ petit.}
\end{equation}
Cela nous permet de d\'efinir, dans un voisinage de $\gamma_0$, une projection
\begin{equation}
  \label{eq:116}
  p:\cC \longrightarrow \cC_0
\end{equation}
par $p(\gamma_0+t\delta(\gamma_0))=\gamma_0$.

Faisons bri\`evement quelques rappels sur la th\'eorie du degr\'e introduite
par Anderson \cite{And08}. Notons $\cE$ l'espace des modules des m\'etriques
d'Einstein AH sur $M$, modulo l'action des diff\'eomorphismes de $M$
induisant l'identit\'e sur $\partial M$. Un d\'etail technique, \`a nouveau, est
qu'il faut consid\'erer des infinis conformes de classe $C^{k,\alpha}$. Alors
$\cE$ est une vari\'et\'e banachique, et l'application qui \`a une m\'etrique
d'Einstein associe son infini conforme,
\begin{equation}
  \label{eq:117}
  \pi:\cE \longrightarrow \cC,
\end{equation}
est Fredholm d'indice 0. Quand on se limite \`a un ouvert $\tilde
\cC\subset\cC$, tel que $\pi|_{\pi^{-1}(\tilde\cC)}$ soit propre, alors le
nombre de points, compt\'e avec signe, d'une fibre de $\pi$ en une
valeur r\'eguli\`ere de $\tilde \pi$, fournit le degr\'e de $\pi$,
un invariant de chaque composante connexe de $\tilde \cC$.

Un voisinage de $\gamma_0\in \cC_0$ ne peut pas \^etre inclus dans $\tilde
\cC$, puisque nous avons justement construit une famille d'\'el\'ements de
$\cE$ qui ne converge pas dans $\cE$, alors que leurs infinis
conformes convergent vers $\gamma_0$. N\'ecessairement donc, le degr\'e ne peut
\^etre d\'efini que sur une r\'egion
$$ \tilde \cC \subset \cC - \cC_0 . $$
Supposons \`a pr\'esent que le degr\'e soit d\'efini au moins sur une r\'egion
$\tilde \cC:=\cV-\cC_0$, o\`u $\cV\subset\cC$ est un voisinage de $\gamma_0$.
Nous consid\'erons l'application
$$ \tilde \pi := p\circ \pi : \pi^{-1}(\tilde \cC) \longrightarrow \cC_0, $$
qui est manifestement Fredholm, d'indice \'egal \`a $1$. Par le th\'eor\`eme
de Sard-Smale, $\tilde \pi$ admet un ensemble dense de valeurs
r\'eguli\`eres, et si $\gamma_0$ est r\'eguli\`ere, alors $\tilde \pi^{-1}(\gamma_0)$ est une
sous-vari\'et\'e de dimension $1$ dans $\cE$. Or nous disposons d\'ej\`a dans
la fibre $\tilde \pi^{-1}(\gamma_0)$ de la famille de m\'etriques d'Einstein
$g_{t,\upsilon,\gamma_0}$ construite par le th\'eor\`eme \ref{theo:constr-E} pour $\upsilon$
et $t>0$ assez petits. Il en r\'esulte imm\'ediatement que, quand on fait
$\upsilon=0$, la famille \`a un param\`etre $g_{t,0,\gamma_0}$ d\'ecrit d\'ej\`a toute la
fibre $\tilde \pi^{-1}(\gamma_0)$ (pr\`es de $g_0$). Par cons\'equent, la famille
$\{g_{t,\upsilon,\gamma_0}, 0<t<\frac\eta2, |\upsilon|<\epsilon\}$ est incluse dans la famille 
$\{g_{t,0,\gamma_0}, 0<t<\eta\}$. On a ainsi d\'emontr\'e que, g\'en\'eriquement, le
param\`etre $\upsilon$ ne fabrique pas de nouvelles m\'etriques :
\begin{theo}\label{theo:transversalite}
  Soit $\gamma_0\in \cC_0$. Supposons comme dans le th\'eor\`eme
  \ref{theo:infinis-conformes} que $\mu_1(\gamma_0)>0$. Si $\gamma_0$ est une
  valeur r\'eguli\`ere de $\tilde \pi$, alors la variation du param\`etre $\upsilon$
  dans les m\'etriques d'Einstein $g_{t,\upsilon,\gamma_0}$ revient \`a varier
  seulement le param\`etre $t$ de la famille $g_{t,0,\gamma_0}$. Autrement
  dit, dans la famille de m\'etriques d'Einstein construite, pour $t$ et
  $\upsilon$ assez petits, la m\'etrique $g_{t,\upsilon,\gamma_0}$ est enti\`erement
  d\'etermin\'ee par son infini conforme.\qed
\end{theo}

Le th\'eor\`eme a pour cons\'equence imm\'ediate que les m\'etriques d'Einstein
$g_{t,\gamma}$, dont les infinis conformes sont exactement d'un seul c\^ot\'e
de $\cC_0$, contribuent \`a une modification du degr\'e d'Anderson par $1$
\`a travers le mur $\cC_0$. Pour \'etablir compl\`etement ce fait, il
faudrait en outre d\'emontrer que les seules m\'etriques d'Einstein
d\'eg\'en\'erant vers $g_0$ sont les m\'etriques $g_{t,\gamma}$, c'est-\`a-dire sont
proches, au sens du corollaire \ref{cor:obs-E}, du recollement d'une
m\'etrique d'Einstein orbifold avec la m\'etrique d'Eguchi-Hanson.

L'exemple le plus simple est $M=T^*S^2$, d\'esingularisation de
$\setR^4/\setZ_2$, de bord $S^3/\setZ_2$. Une analyse simple de l'exemple de la
famille de m\'etriques AdS-Taub-Bolt \cite{PagPop87} indique que l'une
des m\'etriques de Pedersen $g_0$ se d\'esingularise, et que la famille
construite de d\'esingularisations est du c\^ot\'e $\det \bR^+(p_0)>0$ de
$\cC_0$. Cela confirme bien ce qu'indique Anderson
\cite[\S7]{And08}, selon lequel le degr\'e pour la m\'etrique ronde de
$S^3/\setZ_2$ est nul, ce qui correspond \`a la m\'etrique hyperbolique
orbifold, satisfaisant $\det \bR^+(p_0)<0$. Donc le degr\'e doit \^etre
$1$ ou $0$ selon que $\det\bR^+(p_0)$ est $>0$ ou $<0$.

\bibliography{biblio,biquard}
\bibliographystyle{smfalpha}

\end{document}